\documentclass[a4j,10pt]{article}

\usepackage{amsmath}    
\usepackage{graphicx}   
\usepackage[left=20mm,right=20mm,top=20mm,bottom=20mm]{geometry}
\usepackage{verbatim}   
\usepackage{color}      
\usepackage{hyperref}   

\usepackage{amssymb}
\usepackage{amsthm}
\usepackage{mathrsfs}

\newtheorem{Def}{Definition}[section]
\newtheorem{Eg}[Def]{Example}

\newtheorem{Lem}[Def]{Lemma}
\newtheorem{Thm}[Def]{Theorem}
\newtheorem{Cor}[Def]{Corollary}
\newtheorem{Ass}[Def]{Assumption}
\theoremstyle{definition}
\newtheorem{Rem}[Def]{Remark}

\newcommand{\p}{\mathbb{P}}
\newcommand{\e}{\mathbb{E}}
\newcommand{\real}{\mathbb{R}}
\newcommand{\n}{\mathbb{N}}

\newcommand{\weyl}{\mathbb{W}}

\newcommand{\z}{\mathbb{Z}}

\newcommand{\1}{{\bf 1}}

\newcommand{\rd}{\,\mathrm{d}}
\newcommand{\spn}{\mathrm{Span}}

\allowdisplaybreaks

\begin{document}

\title{
Numerical schemes for radial Dunkl processes
}
\author{
Hoang-Long Ngo\footnote{
Hanoi National University of Education, 136 Xuan Thuy, Cau Giay, Hanoi, Vietnam, email~:~\texttt{ngolong@hnue.edu.vn}
}
$\quad $ and $\quad$ 
Dai Taguchi\footnote{
Department of Mathematics,
Kansai University,
Suita,
Osaka,
564-8680,
Japan,
email~:~\texttt{dai.taguchi.dai@gmail.com}
}
}
\maketitle
\begin{abstract}
We consider the numerical approximation for a class of radial Dunkl processes corresponding to arbitrary (reduced) root systems in $\mathbb{R}^{d}$. This class contains well-known processes such as Bessel processes, Dyson's Brownian motions, and square root of Wishart processes. We propose some semi--implicit and truncated Euler--Maruyama schemes for radial Dunkl processes and study their convergence rate with respect to the $L^{p}$-sup norm.
\end{abstract}


\section{Introduction}\label{Sec_Intr}

We are interested in the numerical approximation for a class of radial Dunkl processes corresponding to arbitrary (reduced) root systems in $\mathbb{R}^{d}$, which can be represented as a solution of the following  SDEs,
\begin{align}
\rd X(t)
&
=
\rd B(t)
+
\sum_{\alpha \in R_{+}}
\frac{k(\alpha)}{\langle \alpha, X(t) \rangle}
\alpha
\rd t,  \quad X(0) = x(0) \in \weyl,
\label{RD_0}
\end{align}
where $B$ is an $d$-dimensional Brownian motion defined on a  probability space $(\Omega, \mathscr{F}, \p)$ with the filtration $(\mathscr{F}_t)_{t \geq 0}$ satisfying the usual condition, $R_+$ is a positive root system in $\mathbb{R}^d$,  $k(\cdot)$ is a multiplicity function defined on $R_+$, and $\weyl$ is a Weyl chamber in $\mathbb R^d$.

Dunkl operators were introduced by Dunkl in \cite{Du89} and have been widely studied in both mathematics and physics.
For example, these operators play a crucial role in studying special functions associated with root systems (\cite{vDi, Du91, LaVi}),  Hecke algebras (\cite{Ki97}), and the Hamiltonian operators of some Calogero-Moser-Sutherland quantum mechanical systems (\cite{BaFo1, BaFo2,BaFo3,vDi,vDiVi, He91, HeSc, Kakei, LaVi, Mo75}).
The Dunkl Laplacian, which is defined via Dunkl operators, is a differential-difference operator version of the classical Laplacian.
In \cite{RoVo98}, R\"osler and Voit introduced Dunkl processes as c\`adl\`ag Markov processes with the infinitesimal generator half of the Dunkl Laplacian.
The Dunkl processes have some good features, such as the martingale property and the scaling property (\cite{GaYo05}).
A radial Dunkl process is a continuous Markov process which is defined as the $W$-radial part of a Dunkl process. In \cite{De09}, by using a method of C\'epa and L\'epingle \cite{CeLe01},  Demni proved that a radial Dunkl process can be represented as a solution of an SDE as in \eqref{RD_0} (see also \cite{Sc07} for the radial Heckman--Opdam process).

The class of radial Dunkl processes contains many well-known processes, such as the Bessel processes, Dyson's Brownian motions, and the square root of the multidimensional Wishart processes.
The square root of the Bessel process,  known as the Cox--Ingersoll--Ross process in mathematical finance, and the square root of the multidimensional Wishart process describe the evolution of spot interest rates in one-dimensional and multidimensional settings, respectively.
The numerical schemes for these processes have been studied by many authors (see  \cite{Al, BeBoDi, BoDi, ChJaMi16, DeNeSz, NeSz}).
Dyson's Brownian motions model non-colliding particle systems and have been studied not only in mathematical physics but also in random matrix theory since they have the same laws as the processes of the ordered eigenvalues of some matrix-valued Brownian motions.
The existence and uniqueness of a strong solution for non-colliding particle systems (which include Dyson's Brownian motions) have been well studied (see, e.g. \cite{CeLe97, CeLe01, GrMa13, GrMa14,  NT20, RoSh}).
However, to the best of our knowledge, there are very few results on the numerical analysis for non-colliding particle systems (\cite{LiMe, NT20, LT20}), despite their practical importance.
The numerical approximation for radial Dunkl processes  \eqref{RD_0} is very challenging since their drift coefficient contains very stiff terms of the form $\frac{1}{\langle \alpha, x \rangle}$.
Moreover, since the radial Dunkl processes always take values in the Weyl chamber $\weyl$, it is expected that the approximate solution also takes values in the same domain. 


The forward Euler-Maruyama approximation scheme is not suitable for equation \eqref{RD_0} since it is not difficult to see that the approximate solution could leave the Weyl chamber $\weyl$ with a positive probability.
In \cite{NT20, LT20}, the authors introduced a semi--implicit Euler--Maruyama scheme and semi--implicit Milstein schemes, which preserve the non-colliding property, for some classes of non-colliding particle systems and established their rates of convergence in  $L^p$-norm.
The keys to handling the stiff terms are the one-sided Lipschitz continuous property of the drift and the existence of the inverse moment $\mathbb{E}[|X_i(t) - X_j(t)|^{-p}]$ (see Hypothesis 2.7 in \cite{NT20}).
However, in \cite{NT20}, the finiteness of the inverse moments was proven only for systems with the dimension $d$ less than a certain threshold (see Remark \ref{Rem_inv_0}). Therefore, the approach in \cite{NT20} can not apply to large particle systems. 

This paper aims to extend the work in \cite{NT20} to $d$-dimensional radial Dunkl processes with drift for any $d \geq 1$.
By using the relation between root systems and harmonic functions (Lemma \ref{Lem_harmonic_0}), it can establish a change of measure formulae based on the Girsanov theorem (Lemma \ref{Girsanov_RD_0}).
These formulae first help to prove the existence and uniqueness of the solution of radial Dunkl equations with drift (Corollary \ref{Lem_Gir_drift_0}) and then help to show some estimates for the inverse moments (Lemma \ref{inv_moment_0}).
In order to force the approximate solution to stay in the Weyl chamber $\weyl$, we first construct the semi-implicit Euler-Maruyama scheme as in \cite{NT20}.
We will show that if $\displaystyle \min_{\alpha \in R_{+}} (k(\alpha)-1/2) \geq 3$ (resp. $\geq 16$) then the approximation scheme converges in $L^{p}$ at the rate of order  $1/2$ (resp. $1$), see Theorem \ref{main_1} and Theorem \ref{main_2}.
In addition, since the equations to be solved in each step of the semi-implicit scheme are very nonlinear and have no explicit solutions, we propose a semi-implicit and truncated Euler-Maruyama scheme that can be efficiently implemented in a computer.

This paper is structured as follows.
In Section \ref{Sec_dunkl}, we recall the definitions and some basic properties of the root system, Weyl chamber, Dunkl process, radial  Dunkl process, and the change of measure formula.
In Section \ref{sec:RDPwD}, we study the existence, uniqueness, and some properties of moments of radial Dunkl processes with drift.
Finally, the main results of this paper are presented in Section \ref{sec_EM}, where we propose two numerical schemes for radial Dunkl processes with drift and study their rates of convergence in $L^p$-norm.

\subsection*{Notations}
In this paper, elements of $\real^d$ are column vectors, and for $x \in \real^d$, we denote $x=(x_1,\ldots,x_d)^{\top}$, where for a matrix $A$, $A^{\top}$ stands for the transpose of $A$.
Let $\langle \cdot, \cdot \rangle$ denote the standard Euclidean inner product and $e_{1},\ldots,e_{d}$ be the standard basis vectors of $\real^d$.
We denote $O(d)$ by the orthogonal group and define the orthogonal reflection with respect to $\alpha \in \real^d \setminus \{0\}$ by
$
\sigma_{\alpha}(x)
=x-\frac{2\langle \alpha,x \rangle}{|\alpha|^2} \alpha,~
x \in \real^d.
$
For a function $f:\real^{d} \to \real$, we denote the gradient of $f$ by $\nabla f:=(\frac{\partial f}{\partial x_1},\ldots,\frac{\partial f}{\partial x_d})^{\top}$ and the Laplacian  by $\Delta f:=\sum_{i=1}^{d} \frac{\partial^2 f}{\partial x_i^2}$.
Let $C_b^{1,2}([0, T] \times \weyl;\real^d)$ be the space of $\real^{d}$-valued functions from $[0, T] \times \weyl$ such that the first derivative in time and the first and second derivatives in space exist and are bounded.
For a finite set $A$, we denote $|A|$ the number of all elements of $A$.
We fix $T>0$ and a probability space $(\Omega, \mathscr{F},\p)$ with a filtration $(\mathscr{F}_t)_{t \in [0,T]}$ satisfying the usual conditions.

\section{Radial Dunkl processes and Examples}\label{Sec_dunkl}

\subsection{Root systems and Weyl chamber} \label{subsec_root}

We give definitions of the root system, Weyl groups, and Weyl chamber. 
For details, we refer the reader to \cite{DuXu, Hum12_2, Hum12_1}.

A {\it root system} $R$ in $\real^{d}$ is a finite set of nonzero vectors in $\real^{d}$ such that

\begin{itemize}
\item[(R1)]
$R \cap\{c \alpha ~;~c \in \real\}=\{\alpha,-\alpha\}$ for any $\alpha \in R$;
\item[(R2)]
$\sigma_{\alpha}(R)=R$ for any $\alpha \in R$.
\end{itemize}
An element of a root system is called root.
Additionally, a root system $R$ is said to be {\it crystallographic} if it holds that
\begin{itemize}
\item[(R3)]
 $c_{\alpha \beta}:=2\langle \alpha,\beta \rangle / |\alpha|^2 \in \z$, for any $\alpha,\beta \in R$.
\end{itemize}

A sub-group $W=W(R)$ of $O(d)$ is called the \emph{Weyl group} generated by a root system $R$, if it is generated by the reflections $\{\sigma_{\alpha}~;~\alpha \in R\}$, that is, $W=\langle \sigma_{\alpha} ~|~ \alpha \in R \rangle$.


Each root system can be written as a disjoint unions $R = R_+ \cup (-R_+)$, where $R_+$ and $-R_+$ are separated by a hyperplane $H_\beta := \{x \in \mathbb R^d : \langle \beta, x \rangle = 0\}$ with $\beta \not \in R$. Such a set $R_+$ is called a \emph{positive root system}.  

If $R$ can be expressed as a disjoint union of non-empty sets $R^{1}$ and $R^{2}$ with $\langle \alpha^{1},\alpha^{2} \rangle=0$, for each $\alpha^{i} \in R^{i}$, $i=1,2$, then $R$ and $W(R)$ are called \emph{decomposable} (otherwise are called \emph{indecomposable} or \emph{irreducible}).
Then $R^{i}$, $i=1,2,$ are root systems, respectively and $W(R)=W(R^{1}) \times W(R^{2}):=\{w_{1}w_{2}~;~w_{i} \in W(R^{i}),~i=1,2\}$.
Moreover, $W(R^{i})$, $i=1,2,$ act on the orthogonal subspaces $\spn(R^{i})$, respectively.
It was shown that every root system $R$ can be uniquely expressed as an orthogonal disjoint union of irreducible root systems (see, e.g., Proposition in \cite{Hum12_2} page 57).
More precisely,  there exist $r \in \n$ and irreducible root systems $R^{i}$, $i=1,\ldots,r$ such that $R^{i} \bigcap R^{j} = \emptyset$ and  $\langle \alpha^{(i)}, \alpha^{(j)} \rangle=0$, for every $\alpha^{(i)} \in R^{i}$ and $\alpha^{(j)} \in R^{j}$, $i \neq j$, and
\begin{align}\label{root_deco}
R
=
\bigcup_{j=1}^{r}
R^{j}.
\end{align}

A function $k:R \to \real$ is called a \emph{multiplicity function} if it is invariant under the natural action of $W$ on $R$, that is, $k(\alpha)=k(\beta)$ when there exists $w\in W$ such that $w \alpha=\beta$.
It follows that $k(\alpha^{j})\equiv k_{j}$ for every $\alpha^{j} \in R^{j}$.
We set $\nu(\alpha):=k(\alpha)-1/2$ and $\nu_{j}:=k_{j}-1/2$.

A connected component of $\real^d \setminus \bigcup_{\alpha \in R} H_{\alpha}$ is called \emph{Weyl chamber} of the root system $R$.
In particular, 
\begin{align*}
\weyl
:=
\{x \in \real^d~|~\langle \alpha, x \rangle>0,~\forall \alpha \in R_{+}\}
\end{align*}
is called the \emph{fundamental Weyl chamber}.

One of the important properties of (reduced) root system on stochastic calculus is that the polynomial function $\prod_{\alpha \in R_{+}} \langle \alpha,x \rangle$ is $\Delta$-harmonic.
More precisely, we have the following lemma.

\begin{Lem}[e.g. Theorem 6.2.6 in \cite{DuXu}]\label{Lem_harmonic_0}
Let $V$ be a finite set of non-zero vectors in $\real^d$.
Then $\Delta \prod_{v \in V} \langle v,x \rangle=0$ if and only if there exist constants $c_{v} \in \real$ for $v \in V$ such that $\{c_{v}v~;~v \in V\}=R_{+}$ for some (reduced) root system in $\real^{d}$ and such that no vector in $V$ is a scalar multiple of another vector in $V$.
\end{Lem}

Lemma \ref{Lem_harmonic_0} leads to the following useful properties of root systems.

\begin{Lem}\label{Lem_harmonic_1}
Let $R$ be a root system in $\real^{d}$.
Then $\overline{\delta}(x):=\log \prod_{\alpha \in R_{+}} \langle \alpha ,x \rangle$, $x \in \weyl$, satisfies the equation
\begin{align}\label{harmonic_0}
\Delta \overline{\delta}(x)
+
\sum_{\alpha \in R_{+}}
\frac{\langle \nabla \overline{\delta}(x),\alpha \rangle}{\langle \alpha,x \rangle}
=0.
\end{align}
Moreover, for any $x \in \weyl$, it holds that 
\begin{align}\label{harmonic_1}
\sum_{\alpha \in R_{+}}
\frac
{|\alpha|^2}
{\langle \alpha,x \rangle^2}
=
\sum_{\alpha, \beta \in R_{+}}
\frac
{\langle \alpha, \beta \rangle}
{\langle \alpha,x \rangle \langle \beta,x \rangle}.
\end{align}
\end{Lem}
\begin{proof}
From Lemma \ref{Lem_harmonic_0}, the alternating polynomial $\pi$ defined by $\pi(x):=\prod_{\alpha \in R_{+}} \langle \alpha,x \rangle$ satisfies $\Delta \pi(x)=0$, and $\frac{\partial}{\partial x_i}\log \pi(x)=\sum_{\alpha \in R_{+}} \frac{\alpha_i}{\langle \alpha,x \rangle}$.
Hence we have
\begin{align}\label{harmonic_2}
\Delta \overline{\delta}(x)
&=
\Delta \log \pi (x)
=
\sum_{i=1}^{d}
\frac
{\partial^2}
{\partial x_i^2}
\log \pi (x)
=
\sum_{i=1}^{d}
\frac
{\partial}
{\partial x_i}
\left(
\frac{1}{\pi (x)} 
\frac{\partial \pi (x)}{\partial x_{i}}
\right)
=
-\frac
{|\nabla \pi (x)|^2}
{\pi(x)^2}\notag\\
&
=
-\langle \nabla \log \pi(x), \nabla \log \pi(x) \rangle
=
-
\sum_{i=1}^{d}
\frac{\partial}{\partial x_i}\overline{\delta}(x)
\sum_{\alpha \in R_{+}}
\frac{\alpha_i}{\langle \alpha,x \rangle}
=
-
\sum_{\alpha \in R_{+}}
\frac
{\langle \nabla \overline{\delta}(x), \alpha \rangle}
{\langle \alpha,x \rangle},
\end{align}
which concludes \eqref{harmonic_0}.
Moreover, $\Delta \overline{\delta}(x)$ can be represented as follows
\begin{align*}
\Delta \overline{\delta}(x)
=
\sum_{i=1}^{d}
\frac
{\partial_{i}}
{\partial x_i}
\sum_{\alpha \in R_{+}}
\frac
{\alpha_{i}}
{\langle \alpha,x \rangle}
=
-
\sum_{i=1}^{d}
\sum_{\alpha \in R_{+}}
\frac
{\alpha_{i}^2}
{\langle \alpha,x \rangle^2}
=
-
\sum_{\alpha \in R_{+}}
\frac
{|\alpha|^2}
{\langle \alpha,x \rangle^2}.
\end{align*}
Thus, it follows from \eqref{harmonic_2} that 
\begin{align*}
\sum_{\alpha \in R_{+}}
\frac
{|\alpha|^2}
{\langle \alpha,x \rangle^2}
=
\sum_{\alpha \in R_{+}}
\frac
{\langle \nabla \overline{\delta}(x), \alpha \rangle}
{\langle \alpha,x \rangle}
=
\sum_{i=1}^{d}
\sum_{\alpha \in R_{+}}
\frac{\alpha_i}{\langle \alpha,x \rangle}
\sum_{\beta \in R_{+}}
\frac{\beta_i}{\langle \beta,x \rangle}
=
\sum_{\alpha, \beta \in R_{+}}
\frac
{\langle \alpha, \beta \rangle}
{\langle \alpha,x \rangle \langle \beta,x \rangle},
\end{align*}
which implies \eqref{harmonic_1}.
\end{proof}

\subsection{Dunkl processes and radial Dunkl processes}\label{subsec_dunkl}

In this section, we recall the definitions of Dunkl process and radial Dunkl process.
For details, we refer the reader to \cite{Du89, ChDeGaRoVoYo08, De09, GaYo05, Sc07}.

For a given vector $\xi \in \real^{d}$, a \emph{Dunkl operators} $T_{\xi}$ on $\real^{d}$ associated with $W$ is a differential-difference operator given by
\begin{align*}
T_{\xi}f(x)
:=
\frac
{\partial f(x)}
{\partial \xi}
+
\sum_{\alpha \in R_{+}}
k(\alpha)
\langle \alpha, \xi \rangle
\frac
{f(x)-f(\sigma_{\alpha}x)}
{\langle \alpha,x\rangle},
\end{align*}
where $\frac{\partial}{\partial \xi}$ is the directional derivative with respect to $\xi$.
A Dunkl process in $\real^d$ is a c\`adl\`ag Markov process with the \emph{infinitesimal generator} $\frac{1}{2}{\Delta_{k}f(x)}:=\frac{1}{2}\sum_{i=1}^{d} T_{\xi_{i}}^2$ for any orthonormal basis $\{\xi_{1},\ldots,\xi_{d}\}$, and it has the following explicit form
\begin{align*}
{\Delta_{k}f(x)}
=
\Delta f(x)
+
2\sum_{\alpha \in R_{+}}
k(\alpha)
\left\{
\frac{\langle \nabla f(x),\alpha \rangle}{\langle x, \alpha \rangle}
+
\frac{f(\sigma_{\alpha}x)-f(x)}{\langle \alpha,x \rangle^2}
\right\}.
\end{align*}

A \emph{radian Dunkl process}  $X=(X(t))_{t\geq 0}$ is a continuous Markov process with the infinitesimal generator $L_{k}^{W}/2$ which is defined by
\begin{align*}
\frac{L_{k}^{W}f(x)}{2}
:=&
\frac{\Delta f(x)}{2}
+
\sum_{\alpha \in R_{+}}
k(\alpha)
\frac{\langle \nabla f(x),\alpha \rangle}{\langle \alpha,x \rangle}.
\end{align*}
Moreover, $X$ can be considered as the  $W$-radial part of the Dunkl process $Y$, that is, for the canonical projection $\pi:\real^d \to \real^d/W$, $X=\pi(Y)$, as identifying the space $\real^d/W$ to Weyl chamber $\weyl$ of the root system $R$ (see \cite{GaYo05}).

The following theorem shows that the radial Dunkl process can be represented as a solution of an SDE.

\begin{Thm}
[\cite{Sc07}, \cite{De09}, Proposition 6.1, Lemma 6.4 and Corollary 6.6 in \cite{ChDeGaRoVoYo08}]
Let $X=(X(t))_{t \geq 0}$ be a radial Dunkl process with $X(0)=x(0) \in \weyl$.
Define $T_0:=\inf\{t>0~|~X(t) \in \partial \weyl\}$.
Suppose that for any $\alpha \in R$, $k(\alpha) \geq 1/2$.
Then $T_0=\infty$ a.s.
Moreover, there exists a Brownian motion $B=(B(t))_{t \geq 0}$ such that $X$ is a unique strong solution to the following SDE
\begin{align}
\rd X(t)
&
=
\rd B(t)
+
\sum_{\alpha \in R_{+}}
\frac{k(\alpha)}{\langle \alpha, X(t) \rangle}
\alpha
\rd t \label{RD_1}.
\end{align}
\end{Thm}

We denote by $X^{\nu}$ a solution of SDE \eqref{RD_1} with $k(\alpha)=\nu(\alpha)+1/2$.

\subsection{Examples} 
\label{subsec_ex}

\subsubsection*{Bessel processes}
The finite set $R:=\{\pm 1\}$ is a root system in $\mathbb R$, and the fundamental Weyl chamber is given by $(0,\infty)$.
The corresponding radial Dunkl process $X$ satisfies the following SDE:
\begin{align*}
\rd X(t)
=
\rd B(t)
+
\frac{k}{X(t)}
\rd t,~
X(0)=x(0) \in (0,\infty).
\end{align*}
Therefore, $X$ is a Bessel process with parameter $k \geq 1/2$.

\subsubsection*{Type $A_{d-1}$}
The finite set $R:=\{e_i-e_j \in \real^d~;~i \neq j\} \subset \{x \in \real^d;\sum_{i=1}^{d}x_i=0\}$ is a root system in $\mathbb R^d$. 
A positive root system is given by $R_{+}=\{e_i-e_j~;~i<j\}$.
The fundamental Weyl chamber $\weyl_{A}$ is given by $\weyl_{A}=\{x \in \real^d~;~x_1>x_2>\cdots>x_d\}$. 
The corresponding a type $A$-radial Dunkl process $X$ satisfies the following SDE:
for each $i=1,\dots,d$,
\begin{align*}
\rd X_i(t)
=
\rd B_i(t)
+
\sum_{j:j \neq i}
\frac{k}{X_i(t)-X_j(t)}
\rd t,~
X(0)=x(0) \in \weyl_{A}.
\end{align*}
Therefore, $X$ is a Dyson's Brownian motion with parameter $k \geq 1/2$.


\subsubsection*{Type $B_{d}$, $C_{d}$, $D_{d}$}
Let $r \in \{0, 1,2 \}$ be fixed.
For each  $r\in \{1,2\}$, the finite set $$R(r):=\{e_i-e_j~;~i \neq j\}\cup \{\mathrm{sign}(j-i)(e_i+e_j)~;~i \neq j\} \cup \{\pm r e_i~;~i=1,\ldots,d\} \subset \real^{d}$$ is a root system.
A positive root system is given by $R_{+}(r)=\{e_i-e_j~;~i<j\}\cup\{e_i+e_j~;~i<j\}\cup\{re_i~;~i=1,\ldots,d\}$.
For $r=0$, the finite set $R(0):=\{e_i-e_j~;~i \neq j\}\cup \{\mbox{sign}(j-i)(e_i+e_j)~;~i \neq j\} \subset \real^{d}$ is a root system.
A positive root system is given by $R_{+}(0)=\{e_i-e_j~;~i<j\}\cup\{e_i+e_j~;~i<j\}$. 
$R(1)$, $R(2)$ and $R(0)$ are corresponding root systems for type $B_{d}$, $C_{d}$ and $D_{d}$, respectively.
The fundamental Weyl chambers $\weyl_{B}$, $\weyl_{C}$ and $\weyl_{D}$ are given by $\weyl_{B}=\weyl_{C}=\{x \in \real^d~;~x_1>x_2>\cdots>x_d>0\}$ and $\weyl_{D}=\{x \in \real^d~;~x_1>x_2>\cdots>|x_d|>0\}$, respectively. 
Therefore type $B$, $C$, $D$-radial Dunkl processs satisfy the following SDE:
for each $i=1,\dots,d$,
\begin{align}\label{RD_B_1}
\rd X_i(t)
=
+\rd B_i(t)
+
\frac{rk}{X_i(t)}
\rd t
+
\sum_{j:j \neq i}
k
\left\{
\frac{1}{X_i(t)-X_j(t)}
+
\frac{1}{X_i(t)+X_j(t)}
\right\}
\rd t,~
X(0)=x(0),
\end{align}
where $x(0) \in \weyl_{B}=\weyl_{C}$ if $r\in \{1,2\}$ and $x(0) \in \weyl_{D}$ if $r=0$.
Note that radial Dunkl processes \eqref{RD_B_1} are related to {\it Wishart processes}, (see e.g. \cite{Bru90,De07,GrMa13,GrMa14,KaTa04,KaTa11,KoOc01}).
Indeed, for any $x \in \weyl_{B}=\weyl_{C}$ or $\weyl_{D}$, it holds that
\begin{align*}
\sum_{j:j \neq i}
\left\{
\frac{1}{x_i-x_j}
+
\frac{1}{x_i+x_j}
\right\}
=
\frac{1}{x_{i}}
\sum_{j:j \neq i}
\left\{
\frac{x_i^2+x_j^2}{x_i^2-x_j^2}
\right\}
+
\frac{d-1}{x_{i}},
\end{align*}
thus we have by using It\^o's formula, the stochastic process $Y_{i}(t):=|X_{i}(t)|^{2}$ satisfies the equation
\begin{align}\label{Wishart_SDE}
Y_{i}(t)
=
|x_{i}(0)|^{2}
+
2
\int_{0}^{t}
\sqrt{Y_{i}(s)}
\rd B_{i}(s)
+
\{1+2k(d-1)+r\}t
+
2k
\sum_{j:j \neq i}
\int_{0}^{t}
\frac
{Y_{i}(s)+Y_{j}(s)}
{Y_{i}(s)-Y_{j}(s)}
\rd s,
\end{align}
which are special cases of Wishart processes.
Note that as an application of the main results, we also propose a numerical scheme for a solution of SDE \eqref{Wishart_SDE}.

\subsection{Change of measure}\label{subsec_abso}

We now recall the change of measure formula based on Girsanov's theorem, which was proved in \cite{Ch06} and \cite{ChDeGaRoVoYo08}.
For the convenience of the reader, we will give a proof below.
Recall that $X^{\nu}$ be a solution of SDE \eqref{RD_1} with $k=\nu+1/2$.

\begin{Lem}[\cite{Ch06} or Proposition 6.6.1 in \cite{ChDeGaRoVoYo08}]\label{Girsanov_RD_0}
Suppose that the multiplicity function $k$ satisfies $k=\nu+1/2 \geq 1/2$.
\begin{itemize}
\item[(i)]
Let $Z=(Z(t))_{t \in [0,T]}$ and $M=(M(t))_{t \in [0,T]}$ be stochastic processes defined by 
\begin{align*}
Z(t)
:=
\exp
\left(
	M(t)-\frac{1}{2} \langle M \rangle (t)
\right),~
M(t)
:=
\sum_{i=1}^{d}
\sum_{\alpha \in R_{+}}
\int_{0}^{t}
\frac{\nu(\alpha)}{\langle \alpha,X^{0}(s) \rangle}
\alpha_{i}
\rd B_i(s).
\end{align*}
Then $Z$ is a martingale and satisfies for any $t \in [0,T]$
\begin{align}\label{Girsanov_RD_1}
Z(t)
=
\prod_{\alpha \in R_{+}}
\frac
{\langle \alpha, X^{0}(t) \rangle^{\nu(\alpha)}}
{\langle \alpha, x(0) \rangle^{\nu(\alpha)}}
\exp
\left(
-\frac{1}{2}
\sum_{\alpha \in R_{+}}
\int_{0}^{t}
\frac
{
\nu(\alpha)^{2}
|\alpha|^{2}
}
{
\langle \alpha, X^{0}(s) \rangle^{2}
}
\rd s
\right).
\end{align}

\item[(ii)]
For any measurable function $g:C([0,T];\weyl) \to \real$,
\begin{align*}
\e[g(X^{\nu})]
=\e[g(X^{0}) Z(T)]
\end{align*}
provided that all the above expectations exist.
\end{itemize}
\end{Lem}
\begin{proof}
The idea of the proof is based on \cite{Yor80}.
We recall that the root system $R$ can be decomposed as an orthogonal disjoint irreducible root systems $R^{j}$, $j=1,\ldots,r$ (see \eqref{root_deco}).
Let $\overline{\delta}_j(x):=\log \prod_{\alpha \in R^{j}_{+}} \langle \alpha, x \rangle$  for $x \in \weyl$.
Then it follows from \eqref{harmonic_0}, $\overline{\delta}_j$ satisfies the equation
\begin{align*}
\Delta \overline{\delta}_j(x)
+
\sum_{\alpha \in R^{j}_{+}}
\frac{\langle \nabla \overline{\delta}_{j}(x),\alpha \rangle}{\langle \alpha,x \rangle}
=0.
\end{align*}
Therefore, applying the differential operator $L_{1/2}^{W}$ to $\overline{\delta}_{j}$, we obtain
\begin{align*}
L_{1/2}^{W} \overline{\delta}_{j}(x)
=
\sum_{\alpha \in R_{+} \setminus R^{j}_{+}}
\frac{\langle \nabla \overline{\delta}_{j}(x),\alpha \rangle}{\langle \alpha,x \rangle}
=
\sum_{\alpha \in R_{+} \setminus R^{j}_{+}}
\sum_{\beta \in R^{j}_{+}}
\frac
{\langle \alpha, \beta \rangle}
{\langle \alpha,x \rangle \langle \beta,x \rangle}
=0,
\end{align*}
where in the last equation, we use the fact that $R_{+}^{i}$ and $R_{+}^{j}$, $i \neq j$ are orthogonal.
Thus $\overline{\delta}_{j}$ is $L_{1/2}^{W}$ harmonic.

Recall that since Weyl groups $W(R^{j})$ generated by a root system $R^{j}$, $j=1,\ldots,r,$ act on the orthogonal subspaces $\spn(R^{j})$, respectively, thus it holds that $k(\alpha^{j})=k_{j}=\nu_{j}+1/2$ for each $\alpha^{j} \in R^{j}$, $j=1,\ldots,r$, and thus by using It\^o's formula, we have
\begin{align*}
&\log
\prod_{\alpha \in R_{+}}
\langle \alpha, X^{0}(t) \rangle^{\nu(\alpha)}
=
\sum_{j=1}^{r}
\nu_{j}
\overline{\delta}_{j}(X^{0}(t))\\
&=
\sum_{j=1}^{r}
\nu_{j}
\left\{
\overline{\delta}_{j}(x(0))
+\int_{0}^{t}
\frac{1}{2}L_{1/2}^{W}\overline{\delta}_{j}(X^{0}(s))
\rd s
+
\sum_{i=1}^{d}
\sum_{\alpha \in R^{j}_{+}}
\int_{0}^{t}
\frac{\alpha_i}{\langle \alpha,X^{0}(s) \rangle}
\rd B_i(s)
\right\}\\
&=
\log
\prod_{\alpha \in R_{+}}
\langle \alpha, x(0) \rangle^{\nu(\alpha)}
+M(t).
\end{align*}
Hence, the stochastic process $M$ satisfies
\begin{align*}
M(t)
&=
\log
\left(
\prod_{\alpha \in R_{+}}
\frac
	{\langle \alpha, X^{0}(t) \rangle^{\nu(\alpha)}}
	{\langle \alpha, x(0) \rangle^{\nu(\alpha)}}
\right).
\end{align*}
By using \eqref{harmonic_1}, for any $x \in \weyl$, we have
\begin{align*}
\sum_{\alpha, \beta \in R_{+}}
\frac
{\nu(\alpha)\nu(\beta)\langle \alpha, \beta \rangle}
{\langle \alpha,x \rangle \langle \beta,x \rangle}
=
\sum_{j=1}^{r}
\nu_{j}^{2}
\sum_{\alpha, \beta \in R_{+}^{j}}
\frac
{\langle \alpha, \beta \rangle}
{\langle \alpha,x \rangle \langle \beta,x \rangle}
=
\sum_{j=1}^{r}
\nu_{j}^{2}
\sum_{\alpha \in R_{+}^{j}}
\frac
{|\alpha|^2}
{\langle \alpha,x \rangle^2}
=
\sum_{\alpha \in R_{+}}
\frac
{\nu(\alpha)^{2}|\alpha|^2}
{\langle \alpha,x \rangle^2}.
\end{align*}
Hence, the quadratic variation $\langle M \rangle$ of $M$ satisfies
\begin{align*}
\langle M \rangle(t)
=
\sum_{\alpha,\beta \in R_{+}}
\int_{0}^{t}
\frac{\nu(\alpha)\nu(\beta)\langle \alpha,\beta \rangle}{\langle \alpha,X^{0}(s) \rangle \langle \beta,X^{0}(s)\rangle}
\rd s
=
\sum_{\alpha \in R_{+}}
\int_{0}^{t}
\frac
{
\nu(\alpha)^{2}
|\alpha|^{2}
}
{
\langle \alpha, X^{0}(s) \rangle^{2}
}
\rd s.
\end{align*}
Thus $Z$ satisfies \eqref{Girsanov_RD_1}.

Now, we prove $Z$ is a martingale.
Since by definition, $Z$ is a local martingale. Thus it is sufficient to prove that a family of random variables $\{Z(\tau)~;~\tau \text{ is a stopping time less than } T\}$ is uniformly integrable (see Proposition 1.7 in chapter IV \cite{ReYo99}).
Let $\tau$ be a stopping time with $\tau \leq T$ and $p>1$.
From the representation \eqref{Girsanov_RD_1}, Schwarz's inequality and Lemma \ref{moment_0} with $\nu=0$ and $b=0$, we have
\begin{align*}
\e[|Z(\tau)|^p]
\leq
\e\left[
\prod_{\alpha \in R_{+}}
\frac
	{
		|\alpha|^{p \nu}
		\sup_{0\leq t \leq T}
		|X^{0}(t)|^{p \nu(\alpha)}
	}
	{
		|\langle \alpha,x(0) \rangle|^{p \nu(\alpha)}
	}
\right]
\leq
C_{p}
\end{align*}
for some $C_p$, independent from $\tau$.
This implies uniformly integrability, and thus, we conclude $Z$ is a martingale.
\end{proof}

\section{Radial Dunkl processes with drift} \label{sec:RDPwD} 
\subsection{The existence and uniqueness}\label{subsec_dunkl_b}

In this section, we consider radial Dunkl processes with drift coefficient $b:[0, T] \times \weyl \to \real^d$, satisfying 
\begin{align}\label{RD_SDE}
X(t)
&
=
x(0)
+
B(t)
+
\sum_{\alpha \in R_{+}}
\int_{0}^{t}
\frac{k(\alpha)}{\langle \alpha, X(s) \rangle}
\alpha
\rd s
+
\int_{0}^{t}
b(s,X(s))
\rd s,~
x(0)\in\weyl,~
t \in [0,T].
\end{align}
We denote by $X^{\nu,b}$ a solution of SDE \eqref{RD_SDE} with $k=\nu+1/2$, in particular $X^{\nu,0}=X^{\nu}$.

For multiplicity function $k:R \to \real$, we define a function $f_{k}:\weyl \to \real^d$ by
\begin{align}\label{def_f}
f_{k}(x)
=(f_{k,1}(x),\ldots,f_{k,d}(x))^{\top}
:=
\sum_{\alpha \in R_{+}}
\frac{k(\alpha)}{\langle \alpha, x \rangle}\alpha.
\end{align}
Then $f_k$ is one-sided Lipschitz continuous on $\weyl$.
Indeed, for any $x,y \in \weyl$, it holds that
\begin{align}\label{f_OSL}
\langle x-y, f_{k}(x)-f_{k}(y) \rangle
&=
\sum_{i=1}^{d}
(x_i-y_i)
\sum_{\alpha \in R_{+}}
k(\alpha) \alpha_{i}
	\left\{
		\frac{1}{\langle \alpha, x \rangle}
		-
		\frac{1}{\langle \alpha, y \rangle}
	\right\} \notag\\
&=
-
\sum_{\alpha \in R_{+}}
k(\alpha)
\frac
{\langle \alpha, x-y \rangle^2}
{\langle \alpha, x \rangle \langle \alpha, y \rangle}
\leq 0.
\end{align}

Now we prove the existence and uniqueness of radial Dunkl processes with bounded drift \eqref{RD_SDE} by using Girsanov's theorem and the one-sided Lipschitz property of $f_k$.

\begin{Thm}\label{Lem_Gir_drift_0}
Let $T>0$ be fixed.
Suppose that the multiplicity function $k$ satisfies $k=\nu+1/2 \geq 1/2$ and $b:[0,T] \times \weyl \to \real^d$ is bounded measurable.

\begin{itemize}
\item[(i)]
There exists a weak solution in $\weyl$, and uniqueness in law holds for the  SDE \eqref{RD_SDE}.
In particular, it holds that for any measurable function $g:C([0,T];\weyl) \to \real$,
\begin{align}\label{Lem_Gir_drift_1}
\e[g(X^{\nu,b})]
=
\e[g(X^{\nu}) \widetilde{Z}_1(T)],
\end{align}
provided that all the above expectations exist, where for $q \in \real$, $\widetilde{Z}_q=(\widetilde{Z}_q(t))_{t \in [0,T]}$ is a martingale defined by
\begin{align*}
\widetilde{Z}_q(t)
:=
\exp
\left(
	q \widetilde{M}(t)-\frac{q^2}{2} \langle \widetilde{M} \rangle (t)
\right),~
\widetilde{M}(t)
:=
\sum_{i=1}^{d}
\int_{0}^{t}
b_i(s,X^{\nu}(s))
\rd B_i(s).
\end{align*}

\item[(ii)]
If the map $\weyl \ni x \mapsto b(t,x)$ is one-sided Lipschitz, that is, there exists $K>0$ such that for any $x,y \in \weyl$ and $t \in [0, T]$, $\langle x-y, b(t,x)-b(t,y) \rangle \leq K|x-y|^2$, then the pathwise uniqueness holds.

\end{itemize}
\end{Thm}
\begin{proof}
Since $b$ is bounded, $\sup_{0\leq t \leq T}\e[\exp(\frac{q^2}{2}  \langle \widetilde{M} \rangle (t))]<\infty$ for any $q \in \real$, thus by Novikov's criterion, $\widetilde{Z}_q$ is a martingale.
Therefore, the weak existence and uniqueness in law follow from the Girsanov transformation, and \eqref{Lem_Gir_drift_1} holds.

Now, we prove the pathwise uniqueness.
Suppose that $b$ is one-sided Lipschitz, and let $X$ and $Y$ be two solutions of SDE \eqref{RD_SDE} with $X(0)=Y(0)=x(0)$, driven by the same Brownian motion $B$.
Then since both $b(t,\cdot)$ and $f_{\alpha}$ satisfy one-sided Lipschitz condition,  by using It\^o's formula it holds that for any $t \in [0,T]$
\begin{align*}
|X(t)-Y(t)|^2
&=
\int_{0}^{t}
\langle
	X(s)-Y(s),
	\{f_k(X(s))-f_k(Y(s))\}
	+
	\{b(s,X(s))-b(s,Y(s))\}
\rangle
\rd s\\
&\leq
K
\int_{0}^{t}
|X(s)-Y(s)|^2
\rd s.
\end{align*}
Therefore, Gronwall's inequality implies the statement.
\end{proof}

\begin{Eg}[Radial Heckman--Opdam process \cite{Sc07}]
Let $X$ be a radial Heckman--Opdam process with measurable drift $\widetilde{b}:[0,T]\times \weyl_{A} \to \real^{d}$, of the form
\begin{align}\label{hy_SDE_0}
\rd X(t)
=
\rd B(t)
+
\sum_{\alpha \in R_{+}}
\frac{k(\alpha)}{2}
\coth\left(\frac{\langle \alpha, X(t)\rangle}{2}\right)
\alpha
\rd t
+
\widetilde{b}(t,X(t))
\rd t,~
X(0)=x(0) \in \weyl,
\end{align}
(see also \cite{CeLe01,GrMa14} for hyperbolic particle systems).
By Taylor's expansion of $\coth x$ on $(0,\pi)$, it holds that
\begin{align*}
\coth x
=
\frac{e^{x}+e^{-x}}{e^{x}-e^{-x}}
=
\frac{1}{x}
+
\sum_{n=1}^{\infty} \frac{2^{2n}B_{2n}x^{2n-1}}{(2n)!},
\end{align*}
where $B_n$ is the $n$-th Bernoulli number defined as $B_0=0$ and
$
B_n=-\frac{1}{n+1} \sum_{k=0}^{n-1}
\begin{pmatrix}
n+1\\k
\end{pmatrix}
B_k$.
This shows that the map $(0,\infty) \ni x \mapsto\coth x - \frac{1}{x}$ is smooth and bounded.
Therefore, in Corollary \ref{Lem_Gir_drift_0}, we choose the drift coefficient $b$ as
\begin{align*}
b(t,x)
:=
\sum_{\alpha \in R_{+}}
\frac{k(\alpha)}{2}
\left\{
\coth\left(\frac{\langle \alpha, x \rangle}{2}\right)
-
\frac{2}{\langle \alpha, x \rangle}
\right\}
\alpha
+
\widetilde{b}(t,x),
\end{align*}
and then we conclude that if $\widetilde{b}$ is bounded and one-sided Lipschitz, then SDE \eqref{hy_SDE_0} has a unique strong solution valued in $\weyl$.
\end{Eg}

\subsection{Moments and inverse moments}\label{subsec_moment}

\subsubsection*{Moment estimates}

\begin{Lem}\label{moment_0}
Suppose that the multiplicity function $k$ satisfies $k(\alpha)=\nu(\alpha)+1/2 \geq 1/2$ and $b:[0,T] \times \weyl \to \real^d$ is a bounded measurable function.
Then for any $p>0$, there exists $C_p \in (0, +\infty)$ such that
\begin{align*}
\e\left[
\sup_{0\leq t \leq T}
|X^{\nu,b}(t)|^{p}
\right]
\leq
C_p.
\end{align*}
\end{Lem}
\begin{proof}
Let $\tau_N:=\inf\{t>0~;~|X^{\nu,b}_t| \geq N\}$.
Thanks to H\"older's inequality, it is sufficient to prove the statement for $p=2q \geq 4$.
Since $\langle x, f_k(x) \rangle=\sum_{\alpha \in R_{+}} k(\alpha)=:\gamma_{k}$ for any $x \in \weyl$ and $b$ is bounded, by using It\^o's formula and Schwarz's inequality
\begin{align*}
|X^{\nu,b}(t)|^{2}
&=
|x(0)|^{2}
+
2\int_{0}^{t}
\langle
	X^{\nu,b}(s),
	f_k(X^{\nu,b}(s))
	+
	b(s,X^{\nu,b}(s))
\rangle
\rd s
+
dt
+
2\sum_{i=1}^{d}
\int_{0}^{t}
X^{\nu,b}_i(s)
\rd B_i(s)\\
&\leq
|x(0)|^{2}
+
(2\gamma_{k}+d)T
+
2\|b\|_{\infty}
\int_{0}^{t}
|X^{\nu,b}(s)|
\rd s
+
2
\left|
\sum_{i=1}^{d}
\int_{0}^{t}
	X^{\nu,b}_i(s)
\rd B_i(s)
\right|.
\end{align*}
Therefore, by taking supremum and expectation, it follows from Burkholder-Davis-Gundy's inequality that 
\begin{align*}
&
4^{1-q}\e\left[
	\sup_{0\leq u \leq t}|X^{\nu,b}(u \wedge \tau_N)|^{2q}
\right]\\
&\leq
|x(0)|^{2q}
+
(2\gamma_{k}+d)^qT^q
+
(2\|b\|_{\infty})^{q}T^{q-1}
\int_{0}^{t}
\e\left[
	\sup_{0\leq u \leq s}|X^{\nu,b}(u\wedge \tau_N)|^{q}
\right]
\rd s\\
&\quad
+
(2d)^{q}T^{\frac{q}{2}-1}c_q
\int_{0}^{t}
	\e\left[
		\sup_{0\leq u \leq s}
		\left|
			X^{\nu,b}(u \wedge \tau_N)
		\right|^{2q}
	\right]
\rd s.
\end{align*}
Since $|x|^{q} \leq 1+|x|^{2q}$ for any $x \in \real^d$, thus Gronwall's inequality shows that $\e[\sup_{0\leq t \leq T} |X^{\nu,b}(t \wedge \tau_N)|^{2q}] \leq C_p$, for some $C_p\in (0, +\infty),$ which does not depend on $N$.
By taking $N \to \infty$ and applying Fatou's lemma, we conclude the assertion.
\end{proof}

\subsubsection*{Inverse moment estimates}

In order to estimate inverse moment $\displaystyle \sup_{0\leq t \leq T}\e[\langle \alpha, X^{\nu,b}(t) \rangle^{-p}]$ and $\e[\displaystyle \sup_{0\leq t \leq T} \langle \alpha, X^{\nu,b}(t) \rangle^{-p}]$ for some $p>0$, we apply the change of measure formula given in Lemma \ref{Girsanov_RD_0}.

\begin{Lem}\label{inv_moment_0}
Suppose that the multiplicity function $k$ satisfies $k(\alpha)=\nu(\alpha)+1/2 \geq 1/2$.
Let $\alpha \in R_{+}$.
Then, it holds that
\begin{align}
\label{inv_moment_1}
\sup_{0\leq t \leq T}
\e\left[
\langle \alpha, X^{\nu}(t) \rangle^{-p}
\right]
&\leq
\frac
{C_p}
{\prod_{\beta \in R_{+}}\langle \beta, x(0) \rangle^{\nu(\beta)}}
\text{ if } p \in (0,\nu(\alpha)],
\\
\label{inv_sup_moment_1}
\e\left[
\sup_{0\leq t \leq T}
\langle \alpha, X^{\nu}(t) \rangle^{-p}
\right]
&\leq
\frac{C_{p}}{\langle \alpha, x(0) \rangle^{p}}
+
\frac{C_{p}}{\prod_{\beta \in R_{+}}\langle \beta, x(0) \rangle^{\nu(\beta)}}
\text{ if } p \in \Big[2,\frac{\min_{\beta \in R_{+}}\nu(\beta)}{3}\Big],
\end{align}
for some $C_{p}>0$.
Moreover, if $b:[0,T] \times \weyl \to \real^d$ is bounded and measurable, then it holds that
\begin{align}
\sup_{0\leq t \leq T}
\e\left[
\langle \alpha, X^{\nu,b}(t) \rangle^{-p}
\right]
&\leq
\frac
{C_p}
{\prod_{\beta \in R_{+}}\langle \beta, x(0) \rangle^{\frac{p \nu (\beta)}{\nu(\alpha)}}}
\text{ if } p \in (0,\nu(\alpha)),
\label{inv_moment_2}
\\
\e\left[
\sup_{0\leq t \leq T}
\langle \alpha, X^{\nu,b}(t) \rangle^{-p}
\right]
&\leq
\frac{C_{p}}{\langle \alpha, x(0) \rangle^{p}}
+
\frac{C_{p}}{\prod_{\beta \in R_{+}}\langle \beta, x(0) \rangle^{\frac{p\nu(\beta)}{\nu(\alpha)}}}
\text{ if } p \in \Big[2,\frac{\min_{\beta \in R_{+}}\nu(\beta)}{3}\Big).
\label{inv_sup_moment_2}
\end{align}
\end{Lem}

\begin{proof}
Applying Lemma \ref{Girsanov_RD_0} (ii) and using the martingale property of $Z$, we have
\begin{align*}
\e[\langle \alpha, X^{\nu}(t) \rangle^{-p}]
&=
\e[\langle \alpha, X^{0}(t) \rangle^{-p} Z(t)]
\leq
\e\left[
\langle \alpha, X^{0}(t) \rangle^{-p}
\prod_{\beta \in R_{+}}
\frac
{\langle \beta, X^{0}(t) \rangle^{\nu(\beta)}}
{\langle \beta, x \rangle^{\nu(\beta)}}
\right]\\
&\leq
\frac
{|\alpha|^{\nu(\alpha)-p}}
{\prod_{\beta \in R_{+}}\langle \beta, 
x(0)
\rangle^{\nu(\beta)}}
\e\left[
|X^{0}(t)|^{\nu(\alpha)-p}
\prod_{\beta \in R_{+},~\beta \neq \alpha}
|\beta|^{\nu(\beta)}
|X^{0}(t)|^{\nu(\beta)}
\right].
\end{align*}
Applying Lemma \ref{moment_0}, we obtain \eqref{inv_moment_1}.

Now we prove \eqref{inv_sup_moment_1}.
Define $X_{\alpha}^{\nu}(t)=\langle \alpha,X^{\nu}(t) \rangle$. 
and $Y_{\alpha}^{\nu}(t)=X_{\alpha}^{\nu}(t)^{-1}$, $t \in [0,T]$ and $\alpha \in R_{+}$.
Then $X_{\alpha}^{\nu}(t)$ satisfies the following SDE
\begin{align*}
X_{\alpha}^{\nu}(t)
=
\langle \alpha,x(0) \rangle
+
\langle \alpha,B(t) \rangle
+
\sum_{\beta \in R_{+}}
\int_{0}^{t}
\frac{k(\beta)\langle \alpha,\beta \rangle}{X_{\beta}^{\nu}(s)}
\rd s.
\end{align*}
Since the solution $X$ does not hit the boundary of the Weyl chamber, it follows from It\^o's formula that $Y_{\alpha}^{\nu}(t)$ satisfies the following SDE
\begin{align*}
Y_{\alpha}^{\nu}(t)
&=
\langle \alpha,x(0) \rangle^{-1}
-
\int_{0}^{t}
Y_{\alpha}^{\nu}(s)^{2}
\langle \alpha, \rd B(s) \rangle
+
|\alpha|^{2}
\int_{0}^{t}
Y_{\alpha}^{\nu}(s)^{3}
\rd s
\\&\quad
-
\sum_{\beta \in R_{+}}
k(\beta)
\langle \alpha,\beta \rangle
\int_{0}^{t}
Y_{\alpha}^{\nu}(s)^{2}
Y_{\beta}^{\nu}(s)
\rd s.
\end{align*}
Let $p \in [2,\min_{\beta \in R_{+}}\nu(\beta)/3]$. 
By applying the inequality  $|\sum_{i=1}^m a_i|^p \leq m^{p-1}\sum_{i=1}^m |a_i|^p$ and H\"older's inequality for Lebesgue integrals, we get 
\begin{align*}
\e\left[
\sup_{0 \leq t \leq T}
Y_{\alpha}^{\nu}(t)^{p}
\right]
\leq
&(|R_+|+3)^{p-1}
\Big\{
\langle \alpha,x(0) \rangle^{-p}
+
\e \left[ \sup_{0 \leq t \leq T} \left | \int_0^t  Y_{\alpha}^{\nu}(s)^{2}
\langle \alpha, \rd B(s) \rangle \right|^p\right] \\
&+
|\alpha|^{2p}
T^{p-1}
\int_{0}^{T}
\e\left[
Y_{\alpha}^{\nu}(s)^{3p}
\right]
\rd s
+
\sum_{\beta \in R_{+}}
k(\beta)^{p}
|\langle \alpha,\beta \rangle|^{p} T^{p-1} 
\int_{0}^{T}
\e\left[
Y_{\alpha}^{\nu}(s)^{2p}
Y_{\beta}^{\nu}(s)^{p}
\right]
\rd s
\Big\}.
\end{align*}
Then, by using Young's inequality 
 $Y_{\alpha}^{\nu}(s)^{2p}
Y_{\beta}^{\nu}(s)^{p}
\leq \frac 23 |Y_{\alpha}^{\nu}(s)|^{3p}+\frac 13 |Y_{\beta}^{\nu}(s)|^{3p}$, 
 and Burkholder-Davis-Gundy's inequality for the stochastic integral, we have
\begin{align*}
&\e\left[ 
\sup_{0 \leq t \leq T}
Y_{\alpha}^{\nu}(t)^{p}
\right]\leq 
(|R_+|+3)^{p-1}
\Big\{
\langle \alpha,x(0) \rangle^{-p}
+
c_{p}
d^{p-1}
\sum_{i=1}^{d}|\alpha_{i}|^{p}
T^{p/2-1}
\int_{0}^{T}
\e\left[
Y_{\alpha}^{\nu}(s)^{2p}
\right]
\rd s\\
&\hskip 2cm +
|\alpha|^{2p}
T^{p-1}
\int_{0}^{T}
\e\left[
Y_{\alpha}^{\nu}(s)^{3p}
\right]
\rd s
+
\sum_{\beta \in R_{+}}
k(\beta)^{p}
|\langle \alpha,\beta \rangle|^{p} T^{p-1}
\int_{0}^{T}
\left(
\frac{2}{3}
\e\left[
Y_{\alpha}^{\nu}(s)^{3p}
\right]
+
\frac{1}{3}
\e\left[
Y_{\beta}^{\nu}(s)^{3p}
\right]
\right)
\rd s
\Big\},
\end{align*}
where $c_{p}$ is the constant of Burkholder-Davis-Gundy's inequality. It then follows from \eqref{inv_moment_1} that 
\begin{align*}
\e\left[
\sup_{0 \leq t \leq T}
Y_{\alpha}^{\nu}(t)^{p}
\right]
\leq
C_{p}
\left\{
\frac{1}{\langle \alpha, x(0) \rangle^{p}}
+
\frac{1}{\prod_{\beta \in R_{+}}\langle \beta, x(0) \rangle^{\nu(\beta)}}
\right\},
\end{align*}
for some $C_{p}>0$.
This implies \eqref{inv_sup_moment_1}.

Now we prove \eqref{inv_moment_2}.
Let $p \in [0,\nu(\alpha))$.
By using Theorem \ref{Lem_Gir_drift_0} (i) with $g(w)=\langle \alpha, w(t) \rangle^{-p}$, $w \in C([0,T];\weyl)$, and H\"older's inequality, we get 
\begin{align*}
\e[
\langle \alpha, X^{\nu,b}(t) \rangle^{-p}
]
&
=
\e[
\langle \alpha, X^{\nu}(t) \rangle^{-p}
\widetilde{Z}_{1}(T)
]
\leq
\e[
\langle \alpha, X^{\nu}(t) \rangle^{-\nu(\alpha)}
]^{\frac{p}{\nu(\alpha)}}
\e[
\widetilde{Z}_{1}(T)^{\frac{\nu(\alpha)}{\nu(\alpha)-p}}
]^{\frac{\nu(\alpha)-p}{\nu(\alpha)}}.
\end{align*}
Note that for any $q>1$,
\begin{align*}
\widetilde{Z}_1(t)^{q}
=
\widetilde{Z}_q(t)
\exp
\left(
\frac{q(q-1)}{2}
\langle \widetilde{M} \rangle(t)
\right)
\leq
\exp
\left(
\frac{q(q-1)}{2}
\|b\|_{\infty}^2
T
\right)
\widetilde{Z}_q(t).
\end{align*}
Therefore, since $\widetilde{Z}_q$ is a martingale, by using \eqref{inv_moment_1}, we obtain \eqref{inv_moment_2}.

The proof of \eqref{inv_sup_moment_2} is similar.
This completes the proof of Lemma \ref{inv_moment_0}.
\end{proof}

\begin{Rem}\label{Rem_inv_0}
For Bessel processes $X$ with parameter $k=\nu+1/2$, it holds that $\sup_{0 \leq t \leq T} \e[X(t)^{-p}]<\infty$ for $p \in (0,2\nu+2)$, (e.g. (13) in \cite{Al}, (3.1) in \cite{DeNeSz} or (32) in \cite{NeSz} for CIR processes).
The idea of the proof is to use an explicit representation of the expectation by the confluent hypergeometric function $_{1}F_{1}$ (see Theorem 3.1 in \cite{HuKu08}).
For non-colliding particle systems $\rd X_{i}(t)=\sum_{j \not = i}\frac{k}{X_{i}(t)-X_{j}(t)}\rd t+\sum_{j=1}^d\sigma_{i,j}(X(t))\mathrm{d}B_{j}(t)$ with bounded Lipschitz continuous diffusion coefficient $\sigma$, it is proven in \cite{NT20}  that $\sup_{0 \leq t \leq T} \e[(X_{i}(t)-X_{j}(t))^{-p}]<\infty$ for $p \in (0,\frac{3k}{d\sigma_{d}^{2}}-1]$, where $\sigma_{d}^{2}:=\max_{i=1,\ldots,d} \sup_{x \in \real^{d}} \sum_{k=1}^{d} \sigma_{i,k}(x)^{2}<\infty$ (see Lemma 3.4 in \cite{NT20}). This means that the approach in \cite{NT20}, in which the main idea is to use the It\^o's formula for the function $(x_{i}-x_{j})^{-p}$, is not suitable for systems with a high dimension $d$. 
\end{Rem}

\subsubsection*{Kolmogorov type condition}
As an application of the estimate of the inverse moment, we get the following Kolmogorov type condition of $X^{\nu,b}$.

\begin{Lem}\label{Lem_Kol_0}
Suppose that the multiplicity function $k$ satisfies $k(\alpha)=\nu(\alpha)+1/2 \geq 1/2$ and $b$ is bounded measurable.
If $b=0$ (resp. $b \neq 0$), then for any $p \in (0,\min_{\alpha \in R_{+}}\nu(\alpha)]$ (resp. $p \in (0,\min_{\alpha \in R_{+}}\nu(\alpha))$), there exists $C_p>0$ such that for any $t,s \in [0,T]$,
\begin{align*}
\e\left[
\left|
X^{\nu,b}(t)
-
X^{\nu,b}(s)
\right|^{p}
\right]
\leq
C_p
|t-s|^{p/2}.
\end{align*}
\end{Lem}
\begin{proof}
Let $t,s \in [0, T]$ with $s<t$.
Then, since $b$ is bounded, we have
\begin{align*}
&|X^{\nu,b}(t)-X^{\nu,b}(s)|^{p}\\
&\leq
3^{p-1}
\left\{
	|B(t)-B(s)|^{p}
+
	|R_{+}|^{p-1}
	|t-s|^{p-1}
	\sum_{\alpha \in R_{+}}
		k(\alpha)^p
		\int_{s}^{t}
			\frac{|\alpha|^{p}}{\langle \alpha, X^{\nu,b}(u) \rangle^p}
		\rd u
+
	\|b\|_{\mathrm{Lip}}^p
	|t-s|^{p}
\right\}.
\end{align*}
Hence by taking the expectation, and using Lemma \ref{inv_moment_0}, we conclude the assertion.
\end{proof}

\section{Numerical schemes for radial Dunkl processes}\label{sec_EM}

In this section, we consider the numerical scheme for the radial Dunkl process with drift given by  \eqref{RD_SDE}. We will denote $X^{\nu, b}$ by $X$ to simplify our notation.

\subsection{Semi-implicit Euler--Maruyama scheme}\label{subsec_im_EM}
We first propose a semi-implicit Euler--Maruyama scheme for which the approximate solution takes values in the Weyl chamber $\weyl$.
The construction of this scheme is based on the following lemma.

\begin{Lem}\label{Lem_imp_eq}
Let $x \in \real^d$ and $c:R \to (0,\infty)$ be a measurable function.
The  equation
\begin{equation} \label{eqn_xi}
y=x + \sum_{\alpha \in R_{+}} \frac{c(\alpha)}{\langle \alpha, y \rangle}\alpha
\end{equation}
has a unique solution in $\weyl$.
\end{Lem}
The proof of this lemma is similar to the one of Proposition 2.2 in \cite{NT20} and will be omitted.

Now we can define the semi-implicit Euler--Maruyama approximation  $X^{(n)}=(X^{(n)}(t_{\ell}))_{\ell=0}^{n}$ of  $X$ at times $t_\ell = \frac{\ell T}{n}$ as follows: $X^{(n)}(0):=X(0)=x(0)$, and for each $\ell=0,\ldots,n-1$, $X^{(n)}(t_{\ell+1})$ is the unique solution in $\weyl$ of the following equation:
\begin{align*}
X^{(n)}(t_{\ell+1})
&=
X^{(n)}(t_{\ell})
+
\Delta B_{\ell}
+
\left\{
\sum_{\alpha \in R_{+}}
\frac{k(\alpha)}{\langle \alpha, X^{(n)}(t_{\ell+1}) \rangle}
\alpha
+
b(t_{\ell},X^{(n)}(t_{\ell}))
\right\}
\Delta t.
\end{align*}

If the drift coefficient $b$ is Lipschitz continuous in space and $1/2$-H\"older continuous in time
\begin{align*}
\|b\|_{\mathrm{Lip}}
:=
\sup_{x,y\in \weyl,~x\neq y,~t \in [0,T]}
\frac{|b(t,x)-b(t,y)|}{|x-y|}
+
\sup_{x\in \weyl, ~t,s \in [0,T],~t \neq s}
\frac{|b(t,x)-b(s,x)|}{|t-s|^{1/2}}
<\infty,
\end{align*}
then we have the following result on the rate of strong convergence.

\begin{Thm}\label{main_1}
Suppose that $\min_{\alpha \in R_{+}} \nu(\alpha) \geq 3$ and the drift coefficient $b$ is bounded, and is Lipschitz continuous in space and $1/2$-H\"older continuous in time.
For any $p \in [1,\min_{\alpha \in R_{+}} \nu(\alpha)/3)$, there exists $C=C(p,T)>0$ such that for any $n \in \n$,
\begin{align*}
\e
\left[
\sup_{\ell =1,\ldots,n}
\left|
X(t_{\ell})
-
X^{(n)}(t_{\ell})
\right|^{p}
\right]^{1/p}
\leq
\frac{C}{\sqrt{n}}.
\end{align*}
For $b \equiv 0$, the statement holds also for $p=\min_{\alpha \in R_{+}} \nu(\alpha)/3$.
\end{Thm}

If the drift coefficient $b$ is smooth, then the semi-implicit Euler--Maruyama scheme converges at the strong rate of order one.

\begin{Thm}\label{main_2}
Suppose that $\min_{\alpha \in R_{+}} \nu(\alpha) \geq 16$ and $b \in C^{1,2}_{b}([0,T] \times \weyl;\real^d)$.
For any $p \in [1,\min_{\alpha \in R_{+}} \nu(\alpha)/4)$ there exists $C=C(p,T)>0$ such that for any $n \in \n$ with $T/n \leq 1$,
\begin{align*}
\e
\left[
\sup_{\ell =1,\ldots,n}
\left|
	X(t_{\ell})
	-
	X^{(n)}(t_{\ell})
\right|^{p}
\right]^{1/p}
\leq
\frac{C_p}{n}.
\end{align*}
For $b \equiv 0$, the statement holds also for $p=\min_{\alpha \in R_{+}} \nu(\alpha)/4$.
\end{Thm}

\begin{Rem}
Recently, some tamed Euler-Maruyama approximation schemes have been used to approximate the solution of SDEs with one-sided Lipschitz and super-linear growth coefficients (see \cite{HuJeKl12, Sa13, Sa16}).
Although the drift coefficient $f_{k}$ of $X$  is one-sided Lipschitz, it does not satisfy the ``super linear growth" condition. Thus, it might be difficult to use a tamed (forward) Euler--Maruyama scheme in our setting.
\end{Rem}

As an application of Theorem \ref{main_1} and Theorem \ref{main_2}, we obtain the rate of strong convergence for the numerical approximation of Wishart processes.
\begin{Cor}\label{main_3}
Let $X$ be a solution of SDE \eqref{RD_B_1} and let $Y$ be a Wishart process, which is a solution of SDE \eqref{Wishart_SDE}.
Define $Y^{(n)}=(Y^{(n)}(t_{\ell}))_{\ell=0}^{n}$ by $Y_{i}^{(n)}:=|X_{i}^{(n)}(t_{\ell})|^{2}$ for $i=1,\ldots,d$.
Let $\nu:=k-1/2 \geq 1/2$.
For $p\geq 1$, there exists $C=C(p,T)>0$ such that for any $n \in \n$ with $T/n \leq 1$,
\begin{align*}
\e
\left[
\sup_{\ell =1,\ldots,n}
\left|
Y(t_{\ell})
-
Y^{(n)}(t_{\ell})
\right|^{p}
\right]^{1/p}
\leq
\left\{ \begin{array}{ll}
\displaystyle
\frac{C_p}{\sqrt{n}}
&\text{ if } ,  p \in [1,\nu/3) \text{ and } \nu \geq 3,\\
\displaystyle
\frac{C_p}{n}
&\text{ if } ,  p \in [1,\nu/4) \text{ and } \nu \geq 16.
\end{array}\right.
\end{align*}
\end{Cor}

\subsection{Semi-implicit and truncated Euler--Maruyama schemes}\label{subsec_tr_EM}

The solution of the non-linear equation \eqref{eqn_xi} for the general radial Dunkl process may not have an explicit form, and an efficient numerical solution for the general case has not yet been obtained.
Therefore,  as an alternative approach, we propose a semi-implicit and truncated Euler--Maruyama scheme for $X$, which takes values in $\real^{d}$, not in the Weyl chamber $\weyl$, but it can be implemented efficiently in a computer.

Let $\varepsilon>0$ be fixed.
We first consider an SDE approximation $X_{\varepsilon}=(X_{\varepsilon}(t))_{t \geq 0}$ for $X$.
Define $g_{\varepsilon}:\real \to (0,\infty)$ by $g_{\varepsilon}(x):= \frac{1}{ x \vee \varepsilon}$.
Then, it holds that
\begin{align}
&|g_{\varepsilon}(x)-g_{\varepsilon}(y)|
\leq \varepsilon^{-2}|x-y|,~\text{ for any } x,y \in \real,
\label{g_1}
\\
&(x-y)(g_{\varepsilon}(x)-g_{\varepsilon}(y)) 
\leq 0,~ \text{ for any }  x,y \in \real,
\label{g_2}
\\
&x^{-1}-g_{\varepsilon}(x)
\leq
\varepsilon x^{-2},~ \text{ for any } x>0.
\label{g_3}
\end{align}
We recall that for a multiplicity function $k:R \to \real$, $f_{k}:\weyl \to \real^{d}$ is defined in \eqref{def_f}.
Then we define an approximation $f_{k,\varepsilon}:\real^{d} \to \real^{d}$ of $f_{k}$ by
\begin{align*}
f_{k,\varepsilon}(x)
:=
\sum_{\alpha \in R_{+}}
k(\alpha)
g_{\varepsilon}(\langle \alpha,x\rangle)
\alpha.
\end{align*}
Let $L_{k}:=
\Big(
|R_{+}|\sum_{\alpha \in R_{+}}|k(\alpha)|^{2}|\alpha|^{4}
\Big)^{1/2}.$
Then it follows from \eqref{g_1}, \eqref{g_2}, and \eqref{g_3} that
\begin{align}
& |f_{k,\varepsilon}(x)-f_{k,\varepsilon}(y)|
\leq
L_{k}
\varepsilon^{-2}|x-y|,~\text{ for any } x,y \in \real^{d},
\label{f_1}\\
&\langle x, f_{k,\varepsilon}(x) \rangle
\leq
\sum_{\alpha \in R_{+}} k(\alpha),~\text{ for any }  x \in \real^{d},
\label{f_4}\\
&\langle x-y, f_{k,\varepsilon}(x)-f_{k,\varepsilon}(y) \rangle
\leq 0,~ \text{ for any } x,y \in \real^{d},
\label{f_2}\\
&|f_{k}(x)-f_{k,\varepsilon}(x)|^{2}
\leq
\varepsilon^{2}
|R_{+}|
\sum_{\alpha \in R_{+}}
\frac{|k(\alpha)|^{2}|\alpha|^{2}}{\langle \alpha,x \rangle^{4}},~ \text{ for any } x \in \weyl.
\label{f_3}
\end{align}
If the drift coefficient $b$ is bounded and one-sided Lipschitz, then thanks to the global Lipschitz continuity of $f_{k,\varepsilon}$, the following SDE  has a unique strong solution for any $\epsilon >0$,
\begin{align*}
\rd X_{\varepsilon}
=
\rd B(t)
+
f_{k,\varepsilon}(X_{\varepsilon}(t))
\rd t
+
b(t,X_{\varepsilon}(t))
\rd t,~X_{\varepsilon}(0)=X(0)=x(0) \in \weyl.
\end{align*}
Then we prove that $X_{\varepsilon}$ approximates $X$ in $L^{p}$ sense.

\begin{Lem}\label{Lem_app_X}
Suppose that $\min_{\alpha \in R_{+}} \nu(\alpha) \geq 4$ and $b$ is bounded and one-sided Lipschitz.
For any $p \in [1,\min_{\alpha \in R_{+}} \nu(\alpha)/2)$, there exists $C_{p}>0$ such that for any $\varepsilon>0$,
\begin{align*}
\e\left[
\sup_{0 \leq t \leq T}
\left|
X(t)
-
X_{\varepsilon}(t)
\right|^{p}
\right]^{1/p}
\leq
C_{p}\varepsilon.
\end{align*}
For $b \equiv 0$, the statement holds also for $p = \min_{\alpha \in R_{+}} \nu(\alpha)/2$.
\end{Lem}

\begin{proof}
Thanks to H\"older's inequality, it is sufficient to prove the statement for $p \geq 2$.
By using It\^o's formula, estimate \eqref{f_2}, and Young's inequality $2\langle a, b \rangle \leq |a|^{2}+|b|^{2}$, we have
\begin{align*}
\left|
X(t)
-
X_{\varepsilon}(t)
\right|^{2}
&=
2\int_{0}^{t}
\langle X(s)-X_{\varepsilon}(s), f_{k}(X(s))-f_{k,\varepsilon}(X(s)) \rangle
\rd s
\\&\quad
+
2\int_{0}^{t}
\langle X(s)-X_{\varepsilon}(s), f_{k,\varepsilon}(X(s))-f_{k,\varepsilon}(X_{\varepsilon}(s)) \rangle
\rd s
\\&\quad
+
2\int_{0}^{t}
\langle X(s)-X_{\varepsilon}(s), b(s,X(s))-b(s,X_{\varepsilon}(s)) \rangle
\rd s
\\&\leq
2\int_{0}^{t}
|X(s)-X_{\varepsilon}(s)|
\cdot
|f_{k}(X(s))-f_{k,\varepsilon}(X(s))|
\rd s
+
2K\int_{0}^{t}
|X(s)-X_{\varepsilon}(s)|^{2}
\rd s
\\&\leq
(1+2K)
\int_{0}^{t}
|X(s)-X_{\varepsilon}(s)|^{2}
\rd s
+
\int_{0}^{t}
|f_{k}(X(s))-f_{k,\varepsilon}(X(s))|^{2}
\rd s.
\end{align*}
Since $X(t) \in \weyl$, a.s., applying \eqref{f_3} and Lemma \ref{inv_moment_0}, we have
\begin{align*}
\e\left[
\sup_{0 \leq u \leq t}
\left|
X(u)
-
X_{\varepsilon}(u)
\right|^{p}
\right]
&\leq
T^{p/2-1}(1+2K)^{p/2}
\int_{0}^{t}
\e\left[
\sup_{0 \leq u \leq s}
\left|
X(u)
-
X_{\varepsilon}(u)
\right|^{p}
\right]
\rd s
\\&\quad+
\varepsilon^{p}
T^{p/2-1}
|R_{+}|^{p/2}
\sum_{\alpha \in R_{+}}
|k(\alpha)|^{p}|\alpha|^{p}
\int_{0}^{T}
\e[\langle \alpha,X(s) \rangle^{-2p}]
\rd s
\\&\leq
C_{p}
\int_{0}^{t}
\e\left[
\sup_{0 \leq u \leq s}
\left|
X(u)
-
X_{\varepsilon}(u)
\right|^{p}
\right]
\rd s
+
C_{p}\varepsilon^{p}.
\end{align*}
By using Gronwall's inequality, we obtain the desired result.
\end{proof}

Now, we propose a semi-implicit Euler--Maruyama approximation scheme for $X_{\varepsilon}$.
The construction of such a scheme is based on the following lemma.

\begin{Lem}\label{Lem_tr_eq}
Suppose that  $x \in \real^d$, $\varepsilon>0$ and $h \in (0,  \frac{ \varepsilon^{2}}{L_{k}})$.
Let $F_{k,\varepsilon}(y):=x+hf_{k,\varepsilon}(y)$, and  $y^{(n)}=F_{k,\varepsilon}(y^{(n-1)})$  for any $y^{(0)} \in \real^{d}$ and $n \geq 1$. 
Then the sequence $(y^{(n)})$ converges to  $y^{*} \in \real^{d}$, which is the unique solution of the following  equation
\begin{equation*}
y = x + hf_{k,\varepsilon}(y).
\end{equation*}
In particular, if we choose $y^{(0)}=x$, then it holds that
\begin{align}\label{eq_tr_2}
|y^{*}-y^{(n)}|
\leq
\frac{\sum_{\alpha \in R_{+}}k(\alpha)|\alpha|}{L_{k}(1-L_{k}\varepsilon^{-2}h)}
\cdot \varepsilon\cdot (L_{k}\varepsilon^{-2}h)^{n}.
\end{align}
\end{Lem}

\begin{proof}
It follows from  \eqref{f_1} that 
$|F_{k,\varepsilon}(y)-F_{k,\varepsilon}(y')|
\leq
L_{k}
\varepsilon^{-2}
h
|y - y'|, \text{ for any } y, y' \in \mathbb{R}^d.$
Thus, $F_{k,\varepsilon}$ is a contraction map.
By using the standard argument of the fixed point theorem, we obtain the desired result. 
\end{proof}

The solution is continuous with respect to the truncation parameter. Indeed, let $\varepsilon_1, \varepsilon_2$ be constants satisfying $\varepsilon_1 >  \varepsilon_2 >  \sqrt{hL_k}$. For each $i = 1, 2$, let $y_i \in \mathbb{R}^d$ be a solution  to 
$y = x + h f_{k,\varepsilon_i}(y).$ We have 
\begin{align*}
|y_1- y_2| &\leq h| f_{k,\varepsilon_1}(y_1) - f_{k,\varepsilon_2}(y_1)| + h| f_{k,\varepsilon_2}(y_1) - f_{k,\varepsilon_2}(y_2)| \leq \sum_{\alpha \in R_+} \frac{ |k(\alpha)| |\alpha|}{L_k} |\varepsilon_1 - \varepsilon_2| + h  L_k \varepsilon_2^{-2} |y_1-y_2|.
\end{align*}
Then 
$
|y_1- y_2| \leq \frac{1}{1 -h L_k \varepsilon_2^{-2}}  \sum_{\alpha \in R_+} \frac{ |k(\alpha)| |\alpha|}{L_k} |\varepsilon_1 - \varepsilon_2|.
$

Now thanks to Lemma \ref{Lem_tr_eq}, we can define a semi-implicit truncated Euler--Maruyama approximation scheme $X_{\varepsilon}^{(n)}=(X_{\varepsilon}^{(n)}(t_{\ell}))_{\ell=0}^{n}$ of $X$ as follows: $X_{\varepsilon}^{(n)}(0):=X(0)=x(0)$ and for each $\ell=0,\ldots,n-1$, $X_{\varepsilon}^{(n)}(t_{\ell+1})$ is the unique solution in $\real^{d}$ of the following equation:
\begin{align}\label{IEM_0}
X_{\varepsilon}^{(n)}(t_{\ell+1})
&=
X_{\varepsilon}^{(n)}(t_{\ell})
+
\Delta B_{\ell}
+
\left\{
f_{k,\varepsilon}(X_{\varepsilon}^{(n)}(t_{\ell+1}))
+
b(t_{\ell},X_{\varepsilon}^{(n)}(t_{\ell}))
\right\}
\Delta t.
\end{align}
Note that by Lemma \ref{Lem_tr_eq}, the solution of equation \eqref{IEM_0} can be calculated very quickly using the iteration method.

If the drift coefficient $b$ is Lipschitz continuous in space and $1/2$-H\"older continuous in time, then we have the following result on the rate of strong convergence.

\begin{Thm}\label{main_4}
Suppose that $\min_{\alpha \in R_{+}} \nu(\alpha) > 6$ and the drift coefficient $b$ is bounded, and is Lipschitz continuous in space and $1/2$-H\"older continuous in time.
Assume that $\varepsilon \in (0, \min_{\alpha \in R_{+}} \langle \alpha,x(0) \rangle)$ and $\Delta t=T/n < \varepsilon^{2}/L_{k}$.
Then, for any $p \in [1,\min_{\alpha \in R_{+}} \nu(\alpha)/6)$, there exists $C_{p}>0$ such that for any $n \in \n$,
\begin{align*}
\e
\left[
\sup_{\ell =1,\ldots,n}
\left|
X(t_{\ell})
-
X_{\varepsilon}^{(n)}(t_{\ell})
\right|^{p}
\right]^{1/p}
\leq
C_{p}
\left\{
\varepsilon
+
\frac{1}{\sqrt{n}}
\right\}.
\end{align*}
\end{Thm}

\begin{Rem}
\begin{itemize}
\item[(i)]
An appropriate choice for the parameter $\varepsilon$ is $\varepsilon:=\sqrt{cL_{k}\Delta t}$ for some $c>1$.

\item[(ii)]
Higham, Mao, and Stuart \cite{HMS} considered a backward Euler--Maruyama scheme.
Under one-sided Lipschitz and local Lipschitz conditions for the drift coefficient (see Assumption 3.1, 4.1 in \cite{HMS}), they showed that the scheme converges in $L^{2}$-sup norm at the rate of order $1/2$ (see Theorem 4.4 in \cite{HMS}).
In our setting, the drift coefficient $f_{k,\varepsilon}$ is global Lipschitz continuous, but its Lipschitz constant depends on the parameter $\varepsilon$.
Thus, their result does not directly imply the estimate in Theorem \ref{main_4}.

\end{itemize}
\end{Rem}

\subsection{Proof of main results}\label{subsec_proof}

In order to prove the main results, we first consider a general setting.
Let $D $ be an open subset of $\real^{d}$. 
Let $V=(V(t))_{t \in [0,T]}$ be a solution of the SDE of the form
\begin{align}\label{SDE_general_0}
\rd V(t)
=
\rd B(t)
+
\left\{
g(V(t))
+
b(t,V(t))
\right\}
\rd t,~V(0)=v(0) \in D.
\end{align}

We need the following assumptions for the coefficients $g$ and $b$.

\begin{Ass}\label{Ass_general}
\begin{itemize}
\item[(i)]
$\p(V(t) \in D,~\forall t \in [0,T])=1$;
\item[(ii)]
There exists $K>0$ such that for any $x \in D$, $\langle x,g(x) \rangle \leq K$.
Moreover, for any $x,y \in D$, it holds that $\langle x-y,g(x)-g(y)) \rangle \leq 0$;
\item[(iii)]
$b$ is bounded and Lipschitz continuous in space and $1/2$-H\"older continuous in time;
\item[(iv)]
For any $a \in \real^{d}$ and sufficiently small $h>0$, the system of equation $x=a+hg(x)$ has a unique solution in $D$.
\end{itemize}
\end{Ass}

If $V=X$ (resp. $V=X_{\varepsilon}$), then Assumption \ref{Ass_general} holds  with $D=\weyl$, $g=f_{k}$ (resp. $D=\real^{d}$ and $g=f_{k,\varepsilon})$, and $K=\sum_{\alpha \in R_{+}} k(\alpha)$.
Indeed, it holds that for any $x \in \weyl$, $\langle x,f_{k} \rangle=\sum_{\alpha \in R_{+}} k(\alpha)$ and for any $x \in \real^{d}$, from \eqref{f_4}, $\langle x,f_{k,\varepsilon} \rangle \leq \sum_{\alpha \in R_{+}} k(\alpha)$.

Under Assumption \ref{Ass_general}, we define a semi-implicit Euler--Maruyama scheme $V^{(n)}=(V^{(n)}(t_{\ell}))_{\ell=0}^{n}$ for $V$ as follows: $V^{(n)}(0):=V(0)=v(0)$ and for each $\ell=0,\ldots,n-1$, $V^{(n)}(t_{\ell+1})$ is the unique solution in $D$ of the following equation:
\begin{align}\label{IEM_1}
V^{(n)}(t_{\ell+1})
&=
V^{(n)}(t_{\ell})
+
\Delta B_{\ell}
+
\left\{
g(V^{(n)}(t_{\ell+1}))
+
b(t_{\ell},V^{(n)}(t_{\ell}))
\right\}
\Delta t.
\end{align}
We define the estimation error by $e_{V}(\ell):=V(t_{\ell})-V^{(n)}(t_{\ell})$ for $\ell=0,\ldots,n$.
Then we use the following representation for $e(\ell+1)$:
\begin{align}\label{def_e}
e_{V}(\ell+1)
&=
e_{V}(\ell)
+
\left\{
g(V(t_{\ell+1}))
-
g(V^{(n)}(t_{\ell+1}))
\right\}
\Delta t\notag\\
&
\quad+
\left\{
b(t_{\ell}, V(t_{\ell}))
-
b(t_{\ell}, V^{(n)}(t_{\ell}))
\right\}
\Delta t
+
r_{V}(\ell),
\end{align}
where the reminder part $r_{V}(\ell):=r_{V}(g,\ell)+r_{V}(b,\ell)$ is defined by
\begin{align}
r_{V}(g,\ell)
&:=
\int_{t_{\ell}}^{t_{\ell+1}}
\left\{
g(V(s))
-
g(V(t_{\ell+1}))
\right\}
\rd s,
\label{def_rf}
\\
r_{V}(b,\ell)
&:=
\int_{t_{\ell}}^{t_{\ell+1}}
\left\{
b(s,V(s))-b(t_{\ell},V(t_{\ell}))
\right\}
\rd s.
\label{def_rb}
\end{align}
Then, we obtain the following estimate.

\begin{Lem}\label{Lem_general_0}
Suppose that Assumption \ref{Ass_general} holds.
Then there exists $C_{0}>0$ such that
\begin{align*}
\sup_{\ell =1,\ldots,n}
|e_{V}(\ell)|
\leq
C_{0}
\sum_{j=0}^{n-1}
|r_{V}(j)|,~\text{a.s.}
\end{align*}
\end{Lem}

\begin{proof}
If follows from \eqref{def_e} and Assumption \ref{Ass_general} (ii), (iii) that
\begin{align*}
|e_{V}(\ell+1)|^2
&=
\langle
e_{V}(\ell+1),
e_{V}(\ell)
\rangle
+
\left\langle
e_{V}(\ell+1),
g(V(t_{\ell+1}))
-
g(V^{(n)}(t_{\ell+1}))
\right\rangle
\Delta t\\
&
\quad+
\left\langle
e_{V}(\ell+1),
b(t_{\ell},V(t_{\ell}))
-
b(t_{\ell},V^{(n)}(t_{\ell}))
\right\rangle
\Delta t
+
\langle
e_{V}(\ell+1),
r_{V}(\ell)
\rangle\\
&
\leq
\frac{1}{2}
|e_{V}(\ell+1)|^2
+
\frac{1}{2}
|e_{V}(\ell)|^2
+
\|b\|_{\mathrm{Lip}}
|e_{V}(\ell+1)|
|e_{V}(\ell)|
\Delta t
+
2|e_{V}({\ell+1})|
|r_{V}(\ell)|,
\end{align*}
where the last estimate follows from the Young inequality.
Therefore, we have
\begin{align*}
|e_{V}(\ell+1)|^2
&\leq
|e_{V}(\ell)|^2
+
2\|b\|_{\mathrm{Lip}}
|e_{V}(\ell+1)|
|e_{V}(\ell)|
\Delta t
+
2|e_{V}({\ell+1})|
|r_{V}(\ell)|.
\end{align*}
Thus, by applying the induction method, we derive the following estimate.
\begin{align*}
|e_{V}(\ell+1)|^2
&\leq
2\|b\|_{\mathrm{Lip}}
\sum_{j=0}^{\ell}
|e_{V}(j+1)|
|e_{V}(j)|
\Delta t
+
2
\sum_{j=0}^{\ell}
|e_{V}(j+1)|
|r_{V}(j)|.
\end{align*}
By taking the supremum with respect to $\ell$, we obtain for any $m=1,\ldots,n$
\begin{align*}
\sup_{\ell=1,\ldots, m}
|e_{V}(\ell)|^2
&\leq
2\|b\|_{\mathrm{Lip}}
\sum_{j=0}^{m-1}
|e_{V}(j+1)|
|e_{V}(j)|
\Delta t
+
2
\sum_{j=0}^{m-1}
|e_{V}(j+1)|
|r_{V}(j)|
\\
&
\leq
2\|b\|_{\mathrm{Lip}}
\sup_{\ell=1,\ldots, m}
|e_{V}(\ell)|
\sum_{j=0}^{m-1}
|e_{V}(j)|
\Delta t
+
2
\sup_{\ell=1,\ldots, m}
|e_{V}(\ell)|
\sum_{j=0}^{m-1}
|r_{V}(j)|,
\end{align*}
and thus,
\begin{align*}
\sup_{\ell=1,\ldots, m}
|e_{V}(\ell)|
\leq
2\|b\|_{\mathrm{Lip}}
\sum_{j=0}^{m-1}
\sup_{\ell=1,\ldots, j}
|e_{V}(\ell)|
\Delta t
+
2\sum_{j=0}^{m-1}
|r_{V}(j)|.
\end{align*}
By using discrete Gronwall's inequality (e.g., Chapter XIV, Theorem 1 and Remark 1,2 in \cite{MiPeFi}, page 436-437), we obtain,
\begin{align*}
\sup_{\ell=1,\ldots, n}|e_{V}(\ell)|
&
\leq 
2
\left\{
1
+
\sum_{j=0}^{n-1}
2\|b\|_{\mathrm{Lip}} \Delta t
\exp
\left(
\sum_{j=0}^{n-1}
2\|b\|_{\mathrm{Lip}}
\Delta t
\right)
\right\}
\sum_{j=0}^{n-1}
|r_V(j)|
\leq
C_0
\sum_{j=0}^{n-1}
|r_{V}(j)|,
\end{align*}
where
$
C_0
:=
2
\left\{
1
+
2\|b\|_{\mathrm{Lip}} T
\exp
\left(
	2\|b\|_{\mathrm{Lip}} T
\right)
\right\}
$.
Therefore, we get the assertion.
\end{proof}

\subsubsection*{Proof of Theorem \ref{main_1}}
By using Lemma \ref{Lem_general_0} with $V=X$ and $V^{(n)}=X^{(n)}$, it is sufficient to consider the reminder terms $r_{X}(f_{k},\ell)$ and $r_{X}(b,\ell)$ defined in \eqref{def_rf} and \eqref{def_rb}, respectively.
For each time point $t > 0$, let $\eta_n(t)$ represent the last time discretization point up to time $t$, and let $\kappa_n(t)$ represent the first time discretization point after $t$. More precisely,
 $\eta_{n}(t):=\ell \Delta t$ and $\kappa_{n}(t):= (\ell+1) \Delta t$ if $t \in [\ell \Delta t, (\ell+1) \Delta t)$.
It follows from Lemma \ref{Lem_general_0} that 
\begin{align*}
\sup_{\ell=1,\ldots, n}|e_{X}(\ell)|  &\leq C_0 \sum_{j=0}^{n-1} |r_X(j)| \\ 
&\leq
C_0
\int_{0}^{T}
\left|
f_k(X(s))
-
f_k(X(\kappa_n(s)))
\right|
\rd s
+
C_0
\int_{0}^{T}
\left|
b(s,X(s))
-
b(\eta_n(s),X(\eta_n(s)))
\right|
\rd s.
\end{align*}
If $b=0$ (resp. $b \neq 0$), then for any $p \in [1,\min_{\alpha \in R_{+}}\nu(\alpha)/3]$ (resp. $p \in [1,\min_{\alpha \in R_{+}}\nu(\alpha)/3)$), we have
\begin{align*}
\sup_{\ell=1,\ldots, n}|e_{X}(\ell)|^{p}
&\leq
C_0^{p}T^{p-1}
\int_{0}^{T}
\left|
f_k(X(s))
-
f_k(X(\kappa_n(s)))
\right|^{p}
\rd s
\notag \\&\quad
+
C_0^{p}T^{p-1}
\int_{0}^{T}
\left|
b(s,X(s))
-
b(\eta_n(s),X(\eta_n(s)))
\right|^{p}
\rd s \notag\\
&\leq
C_0^{p}(T \cdot |R_+|)^{p-1}
\sum_{\alpha \in R_{+}}
k(\alpha)^p
\int_{0}^{T}
\frac
{
	|\alpha|^p
	|X(s)-X(\kappa_n(s))|^p
}
{
	\langle \alpha, X(s) \rangle^p
	\langle \alpha, X(\kappa_n(s)) \rangle^p
}
\rd s \notag\\
&\quad
+
\|b\|_{\mathrm{Lip}}^{p}
(2T)^{p-1}
\int_{0}^{T}
\left\{
|s-\eta_n(s)|^{p}
+
|X(s)-X(\eta_n(s))|^{p}
\right\}
\rd s.
\end{align*}
Hence, by taking the expectation and using H\"older's inequality, Lemma \ref{inv_moment_0} and Lemma \ref{Lem_Kol_0}, we have
\begin{align*}
&
\e\left[
\sup_{\ell=1,\ldots, n}|e_{X}(\ell)|^{p}
\right]\\
&\leq
C_0^{p}(T \cdot |R_+|)^{p-1}
\sum_{\alpha \in R_{+}}
k(\alpha)^p
|\alpha|^{p}\\
&\quad\quad \times
\int_{0}^{T}
\e\left[
|X(s)-X(\eta_n(s))|^{3p}
\right]^{1/3}
\e\left[
\langle \alpha, X(s) \rangle^{-3p}
\right]^{1/3}
\e\left[
\langle \alpha, X(\kappa_n(s)) \rangle^{-3p}
\right]^{1/3}
\rd s\\
&\quad
+
\|b\|_{\mathrm{Lip}}^{p}
(2T)^{p-1}
\int_{0}^{T}
\left\{
|s-\eta_n(s)|^{p}
+
\e
\left[
|X(s)-X(\eta_n(s))|^{p}
\right]
\right\}
\rd s\\
&\leq
\frac{C_{p}}{n^{p/2}},
\end{align*}
for some constant $C_{p}>0$.
This concludes the proof.
\qed

\subsubsection*{Proof of Theorem \ref{main_2}}
Before proving Theorem \ref{main_2}, we consider the reminder term $r_{X}(\ell)$, $\ell = 0,\ldots,n$.
We define
\begin{align*}
h_{i}^{(1)}(s,x)
&:=
-
\left(
\left\langle
\nabla f_{k,i}(x),
f_{k}(x)
+
b(s,x)
\right\rangle
+
\frac{\Delta f_{k,i}(x)}{2}
\right),\\
h_{i}^{(2)}(s,x)
&:=
\left(
\left\langle
\nabla b_{i}(s,x),
f_{k}(x)
+
b(s,x)
\right\rangle
+
\partial_s b_{i}(s,x)
+
\frac{\Delta b_{i}(s,x)}{2}
\right ),\\
h_{i}^{(3)}(s,x)
&:=
(h_{i,1}^{(3)}(s,x),\ldots,h_{i,d}^{(3)}(s,x))^{\top}
=
\nabla f_{k,i}(x),\\
h_{i}^{(4)}(s,x)
&:=
(h_{i,1}^{(4)}(s,x),\ldots,h_{i,d}^{(4)}(s,x))^{\top}
=
\nabla b_{i}(s,x),
\end{align*}
for $(s,x) \in [t_{\ell},t_{\ell+1}] \times \weyl$.
By using It\^o's formula, the reminder term $r_{X}(\ell)=(r_{X,1}(\ell),\ldots,r_{X,d}(\ell))^{\top}$ can be decomposed as follows
\begin{align*}
r_{X,i}(\ell)
=
r_{X,i}^{(1)}(\ell)
+
r_{X,i}^{(2)}(\ell)
+
r_{X,i}^{(3)}(\ell)
+
r_{X,i}^{(4)}(\ell),~
i=1,\ldots,d,
\end{align*}
where
$r_{X,i}^{(j)}(\ell):=
\int_{t_{\ell} }^{t_{\ell+1}}
\int_{t}^{t_{\ell+1}}
h^{(j)}_i(s, X_s) 
\rd s
\rd t$
for $j \in \{1, 2\}$, and $ r_{X,i}^{(j)}(\ell)
:= \int_{t_{\ell} }^{t_{\ell+1}}
\int_{t}^{t_{\ell+1}}
\langle
h^{(j)}_i(s, X_s) ,
\rd B(s)
\rangle
\rd t,$
for $ j \in \{3, 4\}$.
We have the following estimates.
\begin{Lem}\label{Lem_Remi_0}
Suppose that $\min_{\alpha \in R_{+}} \nu(\alpha) \geq 8$ and $b \in C^{1,2}_b([0,T] \times \weyl;\real^d)$.
For any $p \in [2,\min_{\alpha \in R_{+}} \nu(\alpha)/4)$, there exists $C_p>0$ such that for all $i=1,\ldots,d$ and $\ell = 0,\ldots, n,$
\begin{align*}
\e[
|r_{X,i}^{(1)}(\ell)|^{p}
]
+
\e[
|r_{X,i}^{(2)}(\ell)|^{p}
]
&\leq
C_p
(\Delta t)^{2p}
\quad\text{and}\quad
\e[
|r_{X,i}^{(3)}(\ell)|^{p}
]
+
\e[
|r_{X,i}^{(4)}(\ell)|^{p}
]
\leq
C_p
(\Delta t)^{\frac{3p}{2}}.
\end{align*}
For $b \equiv 0$, above estimates hold also for $p= \min_{\alpha \in R_{+}} \nu(\alpha)/4$.
\end{Lem}

\begin{proof}
It holds that
\begin{align*}
|h_{i}^{(1)}(s,x)|
&\leq
|\nabla f_{k,i}(x)|
\{
|f_{k}(x)|
+
|b(s,x)|
\}
+
\frac{|\Delta f_{k,i}(x)|}{2}\\
&\leq
\|b\|_{\infty}
|\nabla f_{k,i}(x)|
+
\frac{|f_{k,i}(x)|^2}{2}
+
\frac{|\nabla f_{k,i}(x)|^2}{2}
+
\frac{|\Delta f_{k,i}(x)|}{2},\\
|h_{i}^{(2)}(s,x)|
&\leq
|\nabla b_{i}(x)|
\{
|f_{k}(x)|
+
|b(s,x)|
\}
+
\frac{|\Delta b_{i}(x)|}{2}\\
&\leq
\|\nabla b_i\|_{\infty}
|f_{k}(x)|
+
\|\nabla b_i\|_{\infty}
\|b\|_{\infty}
+
\frac{\|\Delta b_{i}\|_{\infty}}{2},\\
|h_{i}^{(3)}(s,x)|
&\leq
|\nabla f_{k,i}(x)|
,
\quad
|h_{i}^{(4)}(s,x)|
\leq
\|\nabla b_{i}\|_{\infty}.
\end{align*}
Recall that $f_k:\weyl \to \real^d$ is defined by \eqref{def_f}.
Then, the first and second-order derivatives of $f_k$  are given as follows:
\begin{align*}
\frac{\partial f_{k,i}(x)}{\partial x_j}
=
-
\sum_{\alpha \in R_{+}}
k(\alpha)
\frac{\alpha_{i} \alpha_{j} }{\langle \alpha, x \rangle^2}
\quad\text{and}\quad
\frac{\partial^{2} f_{k,i}(x)}{\partial x_j \partial x_m}
=
\sum_{\alpha \in R_{+}}
k(\alpha)
\frac{\alpha_{i} \alpha_{j} \alpha_{m}}{\langle \alpha, x \rangle^3}.
\end{align*}
Therefore, for any $p \geq 2$ by using Jensen' inequality, there exists $K_{p}>0$ such that, for any $x \in \weyl$,
\begin{align*}
|h_{i}^{(1)}(s,x)|^{p}
&\leq
K_{p}
\left(
\sum_{\alpha \in R_{+}}
\frac{1}{\langle \alpha,x \rangle^{2p}}
+
\sum_{\alpha \in R_{+}}
\frac{1}{\langle \alpha,x \rangle^{3p}}
+
\sum_{\alpha \in R_{+}}
\frac{1}{\langle \alpha,x \rangle^{4p}}
\right),\\
|h_{i}^{(2)}(s,x)|^{p}
&\leq
K_{p}
+
K_{p}
\sum_{\alpha \in R_{+}}
\frac{1}{\langle \alpha,x \rangle^{p}},\\
|h_{i}^{(3)}(s,x)|^{p}
&\leq
K_{p}
\sum_{\alpha \in R_{+}}
\frac{1}{\langle \alpha,x \rangle^{2p}},
\quad
|h_{i}^{(4)}(s,x)|^{p}.
\leq
K_{p}
\end{align*}
Hence, from Lemma \ref{inv_moment_0}, there exists $C_p>0$ such that
\begin{align*}
&
\e[|r_{X,i}^{(1)}(\ell)|^{p}]
+
\e[|r_{X,i}^{(2)}(\ell)|^{p}]\\
&\leq
(\Delta t)^{2(p-1)}
\int_{t_{\ell}}^{t_{\ell+1}}\rd t
\int_{t_{\ell}}^{t_{\ell+1}}\rd s 
\e\left[
	|h_{i}^{(1)}(s,X(s))|^{p}
	+|h_{i}^{(2)}(s,X(s))|^{p}
\right]
\leq
C_p
(\Delta t)^{2p},
\end{align*}
and by using Burkholder-Davis-Gundy's inequality,
\begin{align*}
&
\e[|r_{X,i}^{(3)}(\ell)|^{p}
+
\e[|r_{X,i}^{(4)}(\ell)|^{p}]\\
&\leq
(\Delta t)^{p-1}
\int_{t_{\ell} }^{t_{\ell+1}}
\e\left[
	\left|
		\int_{t}^{t_{\ell+1}}
			\left\langle
				h_{i}^{(3)}(s,X(s)),
				\rd B(s)
			\right\rangle
	\right|^{p}
	+
	\left|
		\int_{t_{\ell}}^{t}
			\left\langle
				h_{i}^{(4)}(s,X(s)),
				\rd B(s)
			\right\rangle
	\right|^p
\right]
\rd t \notag\\
&\leq
c_pd^{p-1}
(\Delta t)^{\frac{3p}{2}-2}
\sum_{j=1}^{d}
\int_{t_{\ell} }^{t_{\ell+1}} \rd t
	\int_{t_{\ell} }^{t_{\ell+1}} \rd s
		\e\left[
			|h_{i,j}^{(3)}(X(s))|^{p}
			+
			|h_{i,j}^{(4)}(X(s))|^{p}
		\right]
\leq
C_p
(\Delta t)^{\frac{3p}{2}},
\end{align*}
which concludes the proof.
\end{proof}

\begin{proof}[Proof of Theorem \ref{main_2}]
It is sufficient to prove the statement for $p \geq 4$.
From the representation \eqref{def_e}, we have
\begin{align*}
&
\left|
e_{X}(\ell+1)
-
\left (
f_k(X(t_{\ell+1}))
-
f_k(X^{(n)}(t_{\ell+1}))
\right )
\Delta t
\right|^2\\
&=
\left|
e_{X}(\ell)
+
\left (
b(t_{\ell},X(t_{\ell}))
-
b(t_{\ell},X^{(n)}(t_{\ell}))
\right )
\Delta t
+
r_{X}(\ell)
\right|^2 \notag
\end{align*}
and thus
\begin{align*}
|e_{X}(\ell+1)|^2
&=
|e_{X}(\ell)|^2
-
\left|
f_k(X(t_{\ell+1}))
-
f_k(X^{(n)}(t_{\ell+1}))
\right|^2
(\Delta t)^2
\\
&
\quad
+
\left|
b(t_{\ell},X(t_{\ell}))
-
b(t_{\ell},X^{(n)}(t_{\ell}))
\right|^2
(\Delta t)^2
+|r_{X}(\ell)|^2\\
&
\quad
+
2
\left \langle
e_{X}(\ell+1),
f_k(X(t_{\ell+1}))
-
f_k(X^{(n)}(t_{\ell+1}))
\right \rangle
\Delta t\\
&
\quad
+
2
\left \langle
e_{X}(\ell),
b(t_{\ell},X(t_{\ell}))
-
b(t_{\ell},X^{(n)}(t_{\ell}))
\right \rangle
\Delta t
+
2
\left\langle
e_{X}(\ell), r_{X}(\ell)
\right\rangle\\
&
\quad
+
2
\left
\langle
b(t_{\ell},X(t_{\ell}))
-
b(t_{\ell},X^{(n)}(t_{\ell})),
r_X(\ell)
\right
\rangle
\Delta t.
\end{align*}
We estimate each term of $|e_{X}(\ell+1)|^2$ as follows. First, from \eqref{f_OSL}, we have $\left \langle
e_{X}(\ell+1),
f_k(X(t_{\ell+1}))
-
f_k(X^{(n)}(t_{\ell+1}))
\right \rangle \leq 0$. 
Second, since $b$ is Lipschitz continuous, 
\begin{align*}
&\left| b(t_{\ell},X(t_{\ell})) - b(t_{\ell},X^{(n)}(t_{\ell}))
\right|^2 (\Delta t)^2 
+  2 \left \langle e_{X}(\ell), b(t_{\ell},X(t_{\ell})) - b(t_{\ell},X^{(n)}(t_{\ell})) \right \rangle \Delta t\\
&\leq   \|b\|_{\mathrm{Lip}}^2 |e_{X}(\ell)|^2 (\Delta t)^2  + 
2\|b\|_{Lip} |e_{X}(t_{\ell})|^2 \Delta t. 
\end{align*}
By Young's inequality, 
$$  2\left \langle b(t_{\ell},X(t_{\ell})) - b(t_{\ell},X^{(n)}(t_{\ell})), r_X(\ell) \right \rangle \Delta t  \leq  2\|b\|_{\mathrm{Lip}}
|e_{X}(t_{\ell})| |r_{X}(\ell)| \Delta t \leq 2 \|b\|^2_{\mathrm{Lip}}
|e_{X}(t_{\ell})|^2  (\Delta t)^2 + \frac 12 |r_{X}(\ell)|^2.$$ 
%
%
%
 These estimates, together with the fact that $T/n \leq 1$, imply that 
\begin{align*}
|e_{X}(\ell+1)|^2
&\leq
|e_{X}(\ell)|^2
+
C_1
|e_{X}(\ell)|^2
\Delta t
+
2
\left \langle
e_{X}(\ell),
r_{X}(\ell)
\right \rangle
+
\frac{3}{2}
|r_{X}(\ell)|^2,
\end{align*}
where
$
C_1
:=
3\|b\|_{\mathrm{Lip}}^2
+
2\|b\|_{\mathrm{Lip}}
$.
Thus, we obtain
\begin{align*}
|e_{X}(\ell)|^2
\leq
\sum_{j=0}^{\ell-1}
\left(
C_1
|e_{X}(j)|^2
\Delta t
+
2
\left \langle
	e_{X}(j),
	r_{X}(j)
\right \rangle
+
\frac{3}{2}
|r_{X}(j)|^2
\right).
\end{align*}
Hence for $p=2q\geq 4$ and $m = 0, \ldots, n$, we have
\begin{align}\label{first_1_1}
\sup_{\ell=0,\ldots,m}
|e_{X}(\ell)|^{2q}
&\leq
3^{q-1}
\Big\{
n^{q-1}
\sum_{j=0}^{m-1}
C_1^{q}
|e_{X}(j)|^{2q}
(\Delta t)^{q}
\notag\\&\hspace{1.5cm}
+
2^{2q-1}
\sup_{\ell=0,\ldots,m}
\left (
	|A_{\ell}|^{q}
	+
	|M_{\ell}|^{q}
\right)
+
\frac{3^{q}}{2^{q}}
n^{q-1}
\sum_{j=0}^{m-1}
|r_{X}(j)|^{2q}
\Big\},
\end{align}
where $A_{\ell}:=\sum_{j=0}^{\ell-1}\langle e_{X}(j), r_{X}^{(1)}(j) + r_{X}^{(2)}(j) \rangle$ and $M_{\ell}:=\sum_{j=0}^{\ell-1}\langle e_{X}(jh), r_{X}^{(3)}(j) + r_{X}^{(4)}(j) \rangle$.
Note that it follows from Lemma \ref{inv_moment_0} and the upper bound of $h_{i}^{(3)}(x)$ and $h_{i}^{(4)}(x)$ that
\begin{align*}
\e\left[
M_{\ell}
~\Big|~
\mathscr{F}_{t_{\ell-1}^{(n)}}
\right]
=
M_{\ell-1}
+
\left\langle
e_{X}(\ell-1),
\e\left[
r_{X}^{(3)}(\ell-1)
+
r_{X}^{(4)}(\ell-1)
~\Big|~
\mathscr{F}_{t_{\ell-1}^{(n)}}
\right]
\right\rangle
=M_{\ell-1}.
\end{align*}
Hence $(M_{\ell})_{\ell=1,\ldots,n}$ is a martingale with respect to the filtration $(\mathscr{F}_{t_{\ell}})_{\ell=0}^{n}$.
By using Burkholder-Davis-Gundy's inequality, since $q/2 \geq 1$ we have
\begin{align}\label{first_2}
\e\left[
\sup_{\ell=0,\ldots,m}
|M_{\ell}|^q
\right]
&\leq
c_q
\e\Big[
\Big\{
\sum_{j=0}^{\ell-1}
	|e_{X}(j)|^2 |r_{X}^{(3)}(j) + r_{X}^{(4)}(j)|^2 
\Big\}^{q/2}
\Big] \notag\\
&\leq
2^{q-1}
c_q
n^{\frac{q}{2}-1}
\sum_{j=0}^{\ell-1}
\e\left[
|e_{X}(j)|^{q}
\{
	|r_{X}^{(3)}(j)|^q
	+
	|r_{X}^{(4)}(j)|^{q}
\}
\right].
\end{align}
Therefore, by taking the expectation on \eqref{first_1_1}, we obtain from \eqref{first_2},
\begin{align*}
\e\left[
\sup_{\ell=0,\ldots,m}
|e_{X}(\ell)|^{2q}
\right]
&\leq
(3T)^{q-1}C_1^{q}
\sum_{u=0}^{m-1}
\e\left[
|e_{X}(j)|^{2q}
\right]
\Delta t
+
I_1(m)
+
I_2(m)
+
I_3(m),
\end{align*}
where
\begin{align*}
I_1(m)
&:=
3^{q-1}2^{3q-2}n^{q-1}
\sum_{j=0}^{m-1}
\e\left[
|e_{X}(j)|^{q}
\{
	|r_{X}^{(1)}(j)|^{q}
	+
	|r_{X}^{(2)}(j)|^{q}
\}
\right],
\\
I_2(m)
&:=
3^{q-1}2^{3q-2}
c_q
n^{\frac{q}{2}-1}
\sum_{j=0}^{m-1}
\e\left[
|e_{X}(j)|^{q}
\{
	|r_{X}^{(3)}(j)|^q
	+
	|r_{X}^{(4)}(j)|^{q}
\}
\right],
\\
I_3(m)
&:=
3^{2q-1}2^{-q}
n^{q-1}
\sum_{j=0}^{m-1}
\e[|r_{X}(j)|^{2q}].
\end{align*}
From Schwarz's inequality, the inequality $xy\leq x^2/2+y^2/2$ and Lemma \ref{Lem_Remi_0}, there exists $C_{q}>0$
\begin{align*}
I_{1}(m)
&\leq
3^{q-1}2^{3q-2}n^{q-1}
\sum_{j=0}^{m-1}
\e\left[
|e_{X}(j)|^{2q}
\right]^{1/2}
\left(
\e\left[
	|r_{X}^{(1)}(j)|^{2q}
\right]^{1/2}
+
\e\left[
|r_{X}^{(2)}(j)|^{2q}
\right]^{1/2}
\right)
\\
&
\leq
C_{q}
\sum_{j=0}^{m-1}
\e\left[
|e_{X}(j)|^{2q}
\right]^{1/2}
(\Delta t)^{q+1}
\leq
\frac{C_{q}}{2}
\sum_{j=0}^{m-1}
\e\left[
\sup_{\ell=0,\ldots,j}
|e_{X}(\ell)|^{2q}
\right]
\Delta t
+
\frac{C_{q} T}{2}
(\Delta t)^{2q}.
\end{align*}
and
\begin{align*}
I_2(m)
&\leq
3^{q-1}2^{3q-2}
c_q
n^{\frac{q}{2}-1}
\sum_{j=0}^{m-1}
\e\left[
	|e_{X}(j)|^{2q}
\right]^{1/2}
\left(
\e\left[
	|r_{X}^{(3)}(j)|^{2q}
\right]^{1/2}
+
\e\left[
	|r_{X}^{(4)}(j)|^{2q}
\right]^{1/2}
\right)\\
&\leq
C_{q}
\sum_{j=0}^{m-1}
\e\left[
|e_{X}(j)|^{2q}
\right]^{1/2}
(\Delta t)^{q+1}
\leq
\frac{C_{q}}{2}
\sum_{j=0}^{m-1}
\e\left[
\sup_{\ell=0,\ldots,j}
|e_{X}(\ell)|^{2q}
\right]
\Delta t
+
\frac{C_{q} T}{2}
(\Delta t)^{2q}
\end{align*}
and
\begin{align*}
I_3(m)
&\leq
C_q(\Delta t)^{2q}.
\end{align*}
Therefore, we obtain for some $C>0$,
\begin{align*}
\e\left[
\sup_{\ell=0,\ldots,m}|
e_{X}(\ell)|^{2q}
\right]
\leq
C
\sum_{j=0}^{m-1}
\e\left[
\sup_{\ell=0,\ldots,j}
|e_{X}(\ell)|^{2q}
\right]
\Delta t
+
C(\Delta t)^{2q}.
\end{align*}
By discrete Gronwall's inequality, we conclude the assertion.
\end{proof}

\subsubsection*{Proof of Theorem \ref{main_4}}

For proving Theorem \ref{main_4}, we use the following auxiliary estimate.
Recall that $\eta_{n}(t):=\ell \Delta t$ and $\kappa_{n}(t):= (\ell+1) \Delta t$ if $t \in [\ell \Delta t, (\ell+1) \Delta t)$.

\begin{Lem}\label{Lem_app_f}
Suppose that $\min_{\alpha \in R_{+}} \nu(\alpha)> 6$ and the drift coefficient $b$ is bounded, and is Lipschitz continuous in space and $1/2$-H\"older continuous in time.
Assume that $\varepsilon \in (0, \min_{\alpha \in R_{+}} \langle \alpha,x(0) \rangle)$ and $\Delta t=T/n < \varepsilon^{2}/L_{k}$.
For any $p \in [1,\min_{\alpha \in R_{+}} \nu(\alpha)/6)$ and $q \in [1,\infty)$, there exists $C_{p}>0$ and $C_{q}>0$ such that for any $n \in \n$,
\begin{align*}
\sup_{0\leq t \leq T}
\e\left[
\left|
f_{k,\varepsilon}(X_{\varepsilon}(t))
-
f_{k,\varepsilon}(X_{\varepsilon}(\kappa_n(t)))
\right|^{p}
\right]^{1/p}
&\leq
\frac{C_{p}}{\sqrt{n}},\\
\sup_{0\leq t \leq T}
\e\left[
\left|
b(t,X_{\varepsilon}(t))
-
b(\eta_{n}(t),X_{\varepsilon}(\eta_n(t)))
\right|^{q}
\right]^{1/q}
&
\leq
\frac{C_{q}}{\sqrt{n}}.
\end{align*}
\end{Lem}
\begin{proof}
We first note that since $\sup_{x \in \real^{d}}|f_{k,\varepsilon}(x)|^{2} \leq \sum_{\alpha \in R_{+}}k(\alpha)^{2}|\alpha|^{2} \varepsilon^{-2}$, $b$ is bounded and $\Delta t < \varepsilon^{2}/L_{k}$, for any $q \in [q,\infty)$ and $0\leq s \leq t \leq T$ with $t-s \leq \Delta t$, it holds that
\begin{align}
\label{eq_lem_413}
\e\left[
|X_{\varepsilon}(t)-X_{\varepsilon}(s)|^{q}
\right]
&\leq
3^{q-1}
\e\left[
\left|
B(t)
-
B(s)
\right|^{q}
+
(t-s)^{q-1}
\int_{s}^{t}
|f_{k,\varepsilon}(X_{\varepsilon}(s))|^{q}
\rd s
+
(t-s)^{q-1}
\int_{s}^{t}
|b(s,X_{\varepsilon}(s))|^{q}
\rd s
\right]
\notag
\\&\leq
C_{q}(\Delta t)^{q/2},
\end{align}
for some $C_{q}$.
Therefore, by using the assumptions on the drift coefficient $b$, we obtain
\begin{align*}
\sup_{0\leq t \leq T}
\e\left[
\left|
b(t,X_{\varepsilon}(t))
-
b(\eta_{n}(t),X_{\varepsilon}(\eta_n(t)))
\right|^{q}
\right]^{1/q}
&
\leq
\frac{C_{q}}{\sqrt{n}},
\end{align*}
for some $C_{q}>0$.
Let $p \in [1,\min_{\alpha \in R_{+}} \nu(\alpha)/6)$.
By the definition of $f_{k,\varepsilon}$, it holds that
\begin{align*}
&
\sup_{0 \leq t \leq T}
\e\left[
\left|
f_{k,\varepsilon}(X_{\varepsilon}(t))
-
f_{k,\varepsilon}(X_{\varepsilon}(\kappa_n(t)))
\right|^{p}
\right]
\\&\leq
|R_{+}|^{p-1}
\sum_{\alpha \in R_{+}}
k(\alpha)^{p}
|\alpha|^{p}
\sup_{0 \leq t \leq T}
\e\left[
\left|
g_{\varepsilon}(\langle \alpha, X_{\varepsilon}(t) \rangle)
-
g_{\varepsilon}(\langle \alpha, X_{\varepsilon}(\kappa_n(t)) \rangle)
\right|^{p}
\right].
\end{align*}
We define a stopping times $\tau_{\varepsilon}^{\alpha}:=\inf\{s>0~;~\langle \alpha, X(s) \rangle=\varepsilon\}$ for $\alpha \in R_{+}$ and set $\tau_{\varepsilon}:=\min_{\alpha \in R_{+}} \tau_{\varepsilon}^{\alpha}$.

If $\kappa_{n}(t)<\tau_{\varepsilon}$, then $\min_{\alpha \in R_{+}} \min_{s \in [0,\kappa_{n}(t)]}\langle \alpha,X(s) \rangle >\varepsilon$.
Hence it holds that for any $s \in [0,\kappa_{n}(t)]$, $X(s)=X_{\varepsilon}(s)$ and $g_{\varepsilon}(\langle \alpha,X_{\varepsilon}(s) \rangle)=\langle \alpha,X(s) \rangle^{-1}$.
Therefore, by using Lemma \ref{inv_moment_0} and Lemma \ref{Lem_Kol_0}, we have
\begin{align*}
&\sup_{0 \leq t \leq T}
\e\left[
\left|
g_{\varepsilon}(\langle \alpha, X_{\varepsilon}(t) \rangle)
-
g_{\varepsilon}(\langle \alpha, X_{\varepsilon}(\kappa_n(t)) \rangle)
\right|^{p}
\1_{\{\kappa_{n}(t) <\tau_{\varepsilon}\}}
\right]
\\&=
\sup_{0 \leq t \leq T}
\e\left[
\left|
\frac{1}{\langle \alpha,X(t)\rangle}
-
\frac{1}{\langle \alpha,X(\kappa_n(t))\rangle}
\right|^{p}
\1_{\{\kappa_{n}(t) <\tau_{\varepsilon}\}}
\right]
\leq
\sup_{0 \leq t \leq T}
\e\left[
\frac
{
|\alpha|^p
|X(t)-X(\kappa_n(t))|^p
}
{
\langle \alpha, X(t) \rangle^p
\langle \alpha, X(\kappa_n(t)) \rangle^p
}
\right]
\\&\leq
\sup_{0 \leq t \leq T}
\e\left[
|X(t)-X(\kappa_n(t))|^{3p}
\right]^{1/3}
\e\left[
\langle \alpha, X(t) \rangle^{-3p}
\right]^{1/3}
\e\left[
\langle \alpha, X(\kappa_n(t)) \rangle^{-3p}
\right]^{1/3}
\leq
C_{p}(\Delta t)^{p/2},
\end{align*}
for some $C_{p}>0$.

Now we consider the event $\kappa_{n}(t) \geq \tau_{\varepsilon}$. 
Let $q \in (2p,\min_{\alpha \in R_{+}} \nu(\alpha)/3)$ be fixed.
By using Lipschitz continuity of $f_{k,\varepsilon}$ \eqref{f_1}, H\"older's inequality, Cauchy--Schwarz inequality and \eqref{eq_lem_413}, we have
\begin{align*}
&
\e\left[
\left|
g_{\varepsilon}(\langle \alpha, X_{\varepsilon}(t) \rangle)
-
g_{\varepsilon}(\langle \alpha, X_{\varepsilon}(\kappa_n(t)) \rangle)
\right|^{p}
\1_{\{\kappa_{n}(t) \geq \tau_{\varepsilon}\}}
\right]
\\&\leq
C_{p}
\varepsilon^{-2p}
\e\left[
\left|
\langle \alpha, X_{\varepsilon}(t) \rangle
-
\langle \alpha, X_{\varepsilon}(\kappa_n(t)) \rangle
\right|^{pq/(q-2p)}
\right]^{(q-2p)/q}
\p(\kappa_{n}(t) \geq \tau_{\varepsilon})^{2p/q}
\\&\leq
C_{p}
|\alpha|^{p}
\varepsilon^{-2p}
\e\left[
\left|X_{\varepsilon}(t)-X_{\varepsilon}(\kappa_n(t))\right|^{pq/(q-2p)}
\right]^{(q-2p)/q}
\p(\kappa_{n}(t) \geq \tau_{\varepsilon})^{2p/q}
\\&\leq
C_{p,q}
\varepsilon^{-2p}
(\Delta t)^{p/2}
\p(\kappa_{n}(t) \geq \tau_{\varepsilon})^{2p/q},
\end{align*}
for some $C_{p}>0$ and $C_{p,q}>0$.
Since $\sup_{x \in \real^{d}}|f_{k,\varepsilon}(x)|^{2} \leq \sum_{\alpha \in R_{+}}k(\alpha)^{2}|\alpha|^{2} \varepsilon^{-2}$ and $\Delta t < \varepsilon^{2}/L_{k}$, we have
\begin{align*}
&\e\left[
\left|
\langle \alpha, X_{\varepsilon}(t) \rangle
-
\langle \alpha, X_{\varepsilon}(\kappa_n(t)) \rangle
\right|^{pq/(q-2p)}
\right]
\\&\leq
2^{p-1}
|\alpha|^{p}
\e\left[
\left|
B(t)
-
B(\kappa_{n}(t))
\right|^{pq/(q-2p)}
+
(\Delta t)^{pq/(q-2p) -1}
\int_{\eta_{n}(t)}^{\kappa_{n}(t)}
|f_{k,\varepsilon}(X_{\varepsilon}(s))|^{pq/(q-2p)}
\rd s
\right]
\leq
C_{p,q}(\Delta t)^{pq/(2q-4p)},
\end{align*}
for some $C_{p,q}>0$.
By the definition of the stopping times $\tau_{\varepsilon}^{\alpha}$, $\alpha \in R_{+}$ and \eqref{inv_sup_moment_1} for $b=0$ (resp. \eqref{inv_sup_moment_2} for $b\neq 0$), it follows from Jensen's inequality that
\begin{align*}
\p(\kappa_{n}(t) \geq \tau_{\varepsilon})
&=
\p\left(
\bigcup_{\alpha \in R_{+}}
\left\{
\min_{s \in [0,\kappa_{n}(t)]}
\langle \alpha, X(s) \rangle \leq \varepsilon
\right\}
\right)
\leq
\sum_{\alpha \in R_{+}}
\p\left(
\min_{s \in [0,\kappa_{n}(t)]}
\langle \alpha, X(s) \rangle \leq \varepsilon
\right)
\\&\leq
\varepsilon^{q}
\sum_{\alpha \in R_{+}}
\e\left[
\sup_{s \in [0,T]}
\langle \alpha, X(s) \rangle^{-q}
\right]
\leq
C_{q}
\varepsilon^{q},
\end{align*}
for some $C_{q}>0$.
Therefore we have
\begin{align*}
\sup_{0 \leq t \leq T}
\e\left[
\left|
g_{\varepsilon}(\langle \alpha, X_{\varepsilon}(t) \rangle)
-
g_{\varepsilon}(\langle \alpha, X_{\varepsilon}(\kappa_n(t)) \rangle)
\right|^{p}
\1_{\{\kappa_{n}(t) \geq \tau_{\varepsilon}\}}
\right]
&\leq
C_{p,q}
(\Delta t)^{p/2}.
\end{align*}
We conclude the assertion.
\end{proof}

\begin{proof}[Proof of Theorem \ref{main_4}]
By using Lemma \ref{Lem_app_X} and Lemma \ref{Lem_general_0} with $V=X_{\varepsilon}$ and $V^{(n)}=X_{\varepsilon}^{(n)}$, it is sufficient to consider the reminder terms $r_{X_{\varepsilon}}(f_{k,\varepsilon},\ell)$ and $r_{X_{\varepsilon}}(b,\ell)$ defined in \eqref{def_rf} and \eqref{def_rb}, respectively.
For any $p \in [2,\min_{\alpha \in R_{+}} \nu(\alpha)/6)$, we have
\begin{align*}
\sup_{\ell=1,\ldots, n}|e_{X_{\varepsilon}}(\ell)|^{p}
&\leq
C_0^{p}T^{p-1}
\int_{0}^{T}
\left|
f_{k,\varepsilon}(X_{\varepsilon}(s))
-
f_{k,\varepsilon}(X_{\varepsilon}(\kappa_n(s)))
\right|^{p}
\rd s
\notag\\&\quad+
C_0^{p}T^{p-1}
\int_{0}^{T}
\left|
b(s,X_{\varepsilon}(s))
-
b(\eta_n(s),X_{\varepsilon}(\eta_n(s)))
\right|^{p}
\rd s.
\end{align*}
Hence, by taking the expectation and using Lemma \ref{Lem_app_f}, we have
\begin{align*}
\e\left[
\sup_{\ell=1,\ldots, n}|e_{X_{\varepsilon}}(\ell)|^{p}
\right]
&\leq
\frac{C_{p}}{n^{p/2}},
\end{align*}
for some constant $C_{p}>0$.
This concludes the proof.
\end{proof}

\subsubsection*{Proof of Corollary \ref{main_3}}

We first provide a moment estimate for the semi-implicit Euler--Maruyama schemes $V^{(n)}$.
The proof is based on Lemma 2.5 in \cite{NeSz}.

\begin{Lem}\label{BCD_SDE}
Let $V$ be a solution of SDE \eqref{SDE_general_0} with $b=0$ and let $V^{(n)}$ be the semi-implicit Euler--Maruyama scheme for $V$ defined in \eqref{IEM_1} with $b=0$.
Under Assumption \ref{Ass_general}, for any $p>0$, there exists $C_{p}>0$ such that
\begin{align*}
\e\left[
\sup_{\ell=1,\ldots,n}
\left|V^{(n)}(t_{\ell})\right|^{p}
\right]
\leq C_{p}.
\end{align*}
\end{Lem}
\begin{proof}
Let $v^{*} \in D$ be the unique solution of the equation $v=\Delta t g(v)$.
It follows from Assumption \ref{Ass_general} (ii) that $|v^{*}|\leq \sqrt{\Delta tK}$.
Let $U_{\ell}:=V^{(n)}(t_{\ell})-v^{*}$, $\ell=0,\ldots,n$.
Then by using Assumption \ref{Ass_general} (ii), we have
\begin{align*}
|U_{\ell+1}|^{2}
&=
\langle U_{\ell+1}, U_{\ell}+\Delta B_{\ell}+\Delta t(g(V^{(n)}(t_{\ell+1}))-g(v^{*}))+\Delta t g(v^{*}) \rangle
\\&\leq
\langle U_{\ell+1}, U_{\ell}+\Delta B_{\ell}+\Delta t g(v^{*}) \rangle
\leq
\frac{1}{2}|U_{\ell+1}|^{2}
+
\frac{1}{2}|U_{\ell}+\Delta B_{\ell}+\Delta t g(v^{*})|^{2}.
\end{align*}
Thus by using Young's inequality $\langle a,b \rangle \leq \delta |a|^{2}+(4\delta)^{-1}|b|^{2}$ for $a,b \in \real^{d}$ and $\delta>0$, we have
\begin{align*}
|U_{\ell+1}|^{2}
&\leq
|U_{\ell}+\Delta B_{\ell}+\Delta t g(v^{*})|^{2}
\\&\leq
(1+\Delta t)
|U_{\ell}|^{2}
+
2\langle U_{\ell},\Delta B_{\ell} \rangle
+
2|\Delta B_{\ell}|^{2}
+
|v^{*}|^{2}\Delta t (1+2 \Delta t)
\\&\leq
(1+\Delta t)
|U_{\ell}|^{2}
+
2\langle U_{\ell}, \Delta B_{\ell} \rangle
+
2|\Delta B_{\ell}|^{2}
+
C_{0}\Delta t
\\&\leq
\sum_{j=0}^{\ell}
\left(
2\langle U_{j}, \Delta B_{j} \rangle
+
2|\Delta B_{j}|^{2}
+
C_{0}\Delta t
\right)
(1+\Delta t)^{\ell-j+1},
\end{align*}
for some $C_{0}>0$ independent from $n$.
Hence we have
\begin{align}\label{U_mart}
\sup_{\ell=1,\ldots,n}
|U_{\ell}|^{2}
\leq
e^{T}
\sum_{j=0}^{n-1}
\left(
2\langle U_{j}, \Delta B_{j} \rangle
+
2|\Delta B_{j}|^{2}
+
C_{0}\Delta t
\right).
\end{align}

Now we prove $\e[\sup_{\ell=1,\ldots,n}|U_{\ell}|^{2q}]<\infty$ by induction with respect to $q \in \n$.
For $q=1$, since $U_{j}$ and $\Delta B_{j}$ are independent, it follows from  \eqref{U_mart} that $\e[\sup_{\ell=1,\ldots,n}|U_{\ell}|^{2}]<\infty$.

Suppose that $\e[\sup_{\ell=1,\ldots,n}|U_{\ell}|^{2q}]<\infty$.
Define $M^{1}=(M_{\ell}^{1})_{\ell=0}^{n}$ and $M^{2}=(M_{\ell}^{1})_{\ell=0}^{n}$ by
\begin{align*}
M_{\ell}^{1}
:=
\sum_{j=0}^{\ell}
\langle U_{j}, \Delta B_{j} \rangle
\quad\text{and}\quad
M_{\ell}^{2}
:=
\sum_{j=0}^{\ell}
\left(
|\Delta B_{j}|^{2}
-
d\Delta t
\right).
\end{align*}
Then $M^{1}$ and $M^{2}$ are squared integrable martingales with respect to the filtration $(\mathscr{F}_{t_{\ell}})_{\ell=0}^{n}$.
Hence, by using Doob's inequality, we have
\begin{align*}
\e\left[
\sup_{\ell=1,\ldots,n}
|U_{\ell}|^{4q}
\right]
&\leq
2\cdot 3^{2q-1}e^{T}
\e\left[
\sup_{\ell=1,\ldots,n}
|M_{\ell}^{1}|^{2q}
+
\sup_{\ell=1,\ldots,n}
|M_{\ell}^{2}|^{2q}
\right]
+
Te^{T}3^{2q-1}
(C_{0}+d)^{2q}
\\&\leq
C_{q}
+
C_{q}
\e\left[
\sup_{\ell=1,\ldots,n}
|U_{\ell}|^{2q}
\right]
<\infty,
\end{align*}
for some $C_{q}>0$.
This concludes the assertion.
\end{proof}

\begin{proof}[Proof of Corollary \ref{main_3}]
Let $\delta>0$ be fixed.
Since $|x|^{2}-|y|^{2}=\sum_{i=1}^{d}(x_{i}+y_{i})(x_{i}-y_{i})$, by using H\"older's inequality and Lemma \ref{BCD_SDE} with $V=X$ and $V^{(n)}=X^{(n)}$, we have
\begin{align*}
&\e
\left[
\sup_{\ell =1,\ldots,n}
\left|
Y(t_{\ell})
-
Y^{(n)}(t_{\ell})
\right|^{p}
\right]
\leq
d^{p-1}
\sum_{i=1}^{d}
\e
\left[
\sup_{\ell =1,\ldots,n}
\left|
X_{i}(t_{\ell})
+
X_{i}^{(n)}(t_{\ell})
\right|^{p}
\left|
X_{i}(t_{\ell})
-
X_{i}^{(n)}(t_{\ell})
\right|^{p}
\right]
\\&\leq
d^{p-1}
\sum_{i=1}^{d}
\e
\left[
\sup_{\ell =1,\ldots,n}
\left|
X_{i}(t_{\ell})
+
X_{i}^{(n)}(t_{\ell})
\right|^{\frac{p(1+\delta)}{\delta}}
\right]^{\frac{\delta}{1+\delta}}
\e
\left[
\sup_{\ell =1,\ldots,n}
\left|
X_{i}(t_{\ell})
-
X_{i}^{(n)}(t_{\ell})
\right|^{p(1+\delta)}
\right]^{\frac{1}{1+\delta}}.
\end{align*}
Therefore, by using Theorem \ref{main_1} and Theorem \ref{main_2}, we conclude the assertion.
\end{proof}

\section*{Acknowledgements}
The first author was supported by the NAFOSTED under Grant Number 101.03-2021.36. The second author was supported by JSPS KAKENHI Grant Numbers 17H06833, 19K14552, 21H00988, and 23K12988. 
The authors would like to thank the Vietnam Institute for Advanced Study in Mathematics (VIASM) for its hospitality and support.
The authors also express their gratitude to the referees for their valuable comments, which have significantly contributed to improving the quality of the paper.


\begin{thebibliography}{99}

\bibitem{Al}
{Alfonsi, A.}
{Strong order one convergence of a drift implicit Euler scheme: Application to the CIR process.}
{\it Statist. Probab. Lett.}
{\bf 83}(2)
602--607
(2013).

\bibitem{AGZ10}
{Anderson, G.~W.}, {Guionnet, A.} and {Zeitouni, O.}
{\it An Introduction to Random Matrices.}
volume 118 of {\em Cambridge Studies in Advanced Mathematics}.
Cambridge University Press, Cambridge,
(2010).

\bibitem{AnKr17}
{Andersson, A.} and {Kruse, R.}
{Mean-square convergence of the BDF2-Maruyama and backward Euler schemes for SDE satisfying a global monotonicity condition}.
{\it BIT}
{\bf 57}(1)
21--53
(2017).

\bibitem{BaFo1}
{Baker, T.~H.} and {Forrester, P.~J.}
{The Calogero-Sutherland model and generalized classical polynomials.}
{\it Comm. Math. Phys.}
{\bf 188}
175--216
(1997).

\bibitem{BaFo2}
{Baker, T.~H.} and {Forrester, P.~J.}
{The Calogero-Sutherland model and polynomials with prescribed symmetry.}
{\it Nucl. Phys. B}
{\bf 492}
682--716
(1997).

\bibitem{BaFo3}
{Baker, T.~H.} and {Forrester, P.~J.}
{Non-symmetric Jack polynomials and integral kernels.}
{\it Duke Math.}
{\bf 95}
1--50
(1998).

\bibitem{BeBoDi}
{Berkaoui, A.}, {Bossy, M.}, and {Diop, A.}
{Euler scheme for SDEs with non-Lipschitz diffusion coefficient: strong convergence.}
{\it ESAIM Probab. Stat.}
{\bf 12}
1--11
(2008).

\bibitem{Bia95}
{Biane, P.}
{Permutation model for semi-circular systems and quantum random walks.}
{\it Pacific J. Math.}
{\bf 171}(2)
373--387
(1995).

\bibitem{BoDi}
{Bossy, M.} and {Diop. A.}
{An efficient discretisation scheme for one dimensional SDEs with a diffusion coefficient function of the form $|x|^{\alpha}$, $\alpha \in [1/2,1)$.}
RR-5396
INRIA
D\'ecembre
(2004).

\bibitem{Bru90}
{Bru, M.~F.}
{Wishart processes.}
{\it J. Theoret. Probab.}
{\bf 4}(4)
725--751
(1991).

\bibitem{CeLe97}
{C\'epa, E.} and {L\'epingle, D.}
{Diffusing particles with electrostatic repulsion.}
{\it Probab. Theory Related Fields}
{\bf 107}(4)
429--449
(1997).

\bibitem{CeLe01}
{C\'epa, E.} and {L\'epingle, D.}
{Brownian particles with electrostatic repulsion on the circle:Dyson's model for unitary random matrices revisited.}
{\it ESAIM Probab. Statist.}
{\bf 5}
203--224
(2001).

\bibitem{ChJaMi16}
{Chassagneux, J.~F.}, {Jacquier, A.} and {Mihaylov, I.}
{An explicit Euler scheme with strong rate of convergence for financial SDEs with non-Lipschitz coefficients.}
{\it SIAM J. Financial Math.}
{\bf 7}(1)
993--1021
(2016).

\bibitem{Ch06}
{Chybiryakov, O.}
{\it Processus de {D}unkl et relation de {L}amperti.}
PhD thesis, University Paris 6
(2006).

\bibitem{ChDeGaRoVoYo08}
{Chybiryakov, O.}, {Demni, N.}, {Gallardo, L.}, {R{\"o}sler, M.}, {Voit, M.} and {Yor, M.}
{\it Harmonic and Stochastic Analysis of Dunkl Processes}.
Hermann
(2008).


\bibitem{De07}
{Demni, N.}
{The Laguerre process and generalized Hartman–Watson law.}
{\it Bernoulli}
{\bf 13}(2)
556--580
(2007).


\bibitem{De09}
{Demni, N.}
{Radial Dunkl processes: existence, uniqueness and hitting time.}
{\it C. R. Math. Acad. Sci. Paris, Ser. I}
{\bf 347}
1125--1128
(2009).


\bibitem{DeNeSz}
{Dereich, S.}, {Neuenkirch, A.} and {Szpruch, L.}
{An Euler-type method for the strong approximation for the Cox-Ingersoll-Ross process.}
{\it Proc. R. Soc. A}
{\bf 468}
1105--1115
(2012).

\bibitem{vDi}
{van Diejen, J.~F.}
{Confluent hypergeometric orthogonal polynomials related to the rational quantum Calogero system with harmonic confinement.}
{\it Comm. Math. Phys}
{\bf 188}
467--497
(1997).

\bibitem{vDiVi}
{van Diejen, J.~F.} and {Vinet, L.}
{\it Calogero-Sutherland-Moser Models. CRM.}
Springer
(2000).

\bibitem{Du89}
{Dunkl, C.}
{Differential-difference operators associated to reflection groups.}
{\it Trans. Amer. Math. Soc.}
{\bf 311}(1)
167--183
(1989).

\bibitem{Du91}
{Dunkl, C.}
{Integral kernels with reflection group invariance.}
{\it Can. J. Math.}
{\bf 43}(6)
1214--1227
(1991).

\bibitem{DuXu}
{Dunkl, C.} and {Xu, Y.}
{\it Orthogonal Polynomials of Several Variables, Second Edition.}
Cambridge university press
(2001).

\bibitem{Dyson}
{Dyson, F.~J.}
{A Brownian-motion model for the eigenvalues of a random matrix.}
{\it J. Math. Phys.}
{\bf 3}(6)
1191--1198
(1962).

\bibitem{GaYo05}
{Gallardo, L.} and {Yor, M.}
{Some new examples of Markov processes which enjoy the time-inversion property.}
{\it Probab. Theory Relat. Fields}
{\bf 132}
150--162
(2005).

\bibitem{Gi08}
{Giles, M.~B.}
{Multilevel Monte Carlo path simulation.}
{\it Oper. Res.}
{\bf 56}(3)
607--617
(2008).

\bibitem{Gra99}
{Grabiner, D.~J.}
{Brownian motion in a {W}eyl chamber, non-colliding particles, and random matrices.}
{\it Ann. Inst. H. Poincar{\'e} Probab. Statist.}
{\bf 35}(2)
177--204
1999.

\bibitem{GrMa13}
{Graczyk, P.} and {Ma\l ecki, J.}
{Multidimensional Yamada-Watanabe theorem and its applications to particle systems.}
{\it J. Math. Phys.}
{\bf 54}(2)
021503
(2013).

\bibitem{GrMa14}
{Graczyk, P.} and {Ma\l ecki, J.}
{Strong solutions of non-colliding particle systems.}
{\it Electron. J. Probab.}
{\bf 19}(119)
1--21
(2014).

\bibitem{HaWa91}
{Hairer, E.} and {Wanner, G.}
{\it Solving Ordinary Differential Equations II:  Stiff and Differential-algebraic Problems.}
Springer
(1991).

\bibitem{He91}
{Heckman, G.~J.}
{A remark on the Dunkl differential-difference operators, in Harmonic analysis on reductive groups.}
{\it Progress in Math.}
{\bf 101}
181--191
(1991).

\bibitem{HeSc}
{Heckman, G.~J.} and {Schlichtkrull, H.}
{\it Harmonic Analysis and Special Functions on Symmetric Spaces.}
Perspectives in Mathematics
{\bf 16}
Academic Press
(1994).

\bibitem{HMS}
{Higham,~D.~J.}, {Mao,~X.} and {Stuart,~A.~M.}
{Strong convergence of Euler-type methods for nonlinear stochastic differential equations.}
{\it SIAM J. Numer. Anal.}
{\bf 40}(3)
1041--1063
(2002).

\bibitem{Hu96}
{Hu, Y.}
{Semi-implicit Euler--Maruyama scheme for stiff stochastic equations.}
Stochastic Analysis and Related Topics V:
The Silvri Workshop, Progr. Probab. 38, H. Koerezlioglu,
Birkhauser, Boston
183--202
(1996).

\bibitem{Hum12_2}
{Humphreys, J.~E.}
{\it Introduction to Lie Algebras and Representation Theory.}
Springer
(1972).

\bibitem{Hum12_1}
{Humphreys, J.~E.}
{\it Reflection Groups and Coxeter Groups.}
Cambridge university press
(1992).

\bibitem{HuKu08}
{Hurd, T.~R.} and {Kuznetsov, A.}
{Explicit formulas for Laplace transforms of stochastic integrals.}
{\it Markov Process. Related Fields}
{\bf 14}(2)
277--290
(2008).


\bibitem{HuJeKl12}
{Hutzenthalerm M.}, {Jentzen, A.} and {Kloeden, P.~E.}
{Strong convergence of an explicit numerical method for SDEs with non-globally Lipschitz continuous coefficients.}
{\it Ann. Appl. Probab.}
{\bf 22}(4)
1611--1641
(2012).


\bibitem{Kakei}
{Kakei, S.}
{Common algebraic structure for the Calogero-Sutherland models.}
{\it J. Phys. A}
{\bf 29}
619--624
(1996).



\bibitem{Ka15}
{Katori, M.}
{\it Bessel Processes, {S}chramm-{L}oewner Evolution, and the {D}yson model.}
Springer
(2015).

\bibitem{KaTa04}
{Katori, M.} and {Tanemura, H.}
{Symmetry of matrix-valued stochastic processes and noncolliding diffusion particle systems.}
{\it J. Math. Phys.}
{\bf 45}(8)
3058--3085
(2004).

\bibitem{KaTa11}
{Katori, M.} and {Tanemura, H.}
{Noncolliding squared Bessel processes.}
{\it J. Stat. Phys.}
{\bf 142}(3)
592--615
(2011).

\bibitem{Ki97}
{Kirillov A.~A.}
{Lectures on affine Hecke algebras and Macdonald’s conjectures.}
{\it Bull. Amer> Math. Soc.}
{\bf 34}(3)
251--292
(1997).

\bibitem{KP}
{Kloeden, P.} and  {Platen, E.}
{\it Numerical Solution of Stochastic Differential Equations.}
Springer
(1995).

\bibitem{KoOc01}
{K\"onig, W.} and {O'connell, N.}
{Eigenvalues of the Laguerre process as non-colliding squared Bessel processes.}
{\it Elect. Comm. in Probab.}
{\bf 6}
107--114
(2001).

\bibitem{LaVi}
{Lapointe, L.} and {Vinet, L.}
{Exact Operator Solution of the Calogero-Sutherland Model.}
{\it Comm. Math. Phys.}
{\bf 178}
425--452
(1996).

\bibitem{LiMe}
{Li, X.~H.} and {Menon, G.}
{Numerical solution of Dyson Brownian motion and a sampling scheme for invariant matrix ensembles.}
{\it J. Stat. Phys.}
{\bf 153}(5)
801–-812
(2013).

\bibitem{LT20}
{Luong, D.~T.} and {Ngo, H.~L.}
{Semi-implicit Milstein approximation scheme for non-colliding particle systems.}
{\it Calcolo}
{\bf 56}(25)
(2019).


\bibitem{Mar54}
{Maruyama, G.}
{On the transition probability functions of the Markov process.}
{\it Nat. Sci. Rep. Ochanomizu Univ.}
{\bf 5}
10--20
(1954).

\bibitem{MiPeFi}
{Mitrinovi\'c, D.~S.}, {Pe\v ari\'c, J.~E.} and {Fink, A.~M.}
{\it Inequalities involving functions and their integrals and derivatives.}
Kluwer Academic Publishers
(1994).

\bibitem{Mo75}
{Moser, J.}
{Three integrable Hamiltonian systems connected with isospectral deformations.} {\it Adv. in Math.}
{\bf 16}
197--220
(1975).

\bibitem{NeSz}
{Neuenkirch, A.} and {Szpruch, L.}
{First order strong approximations of scalar SDEs defined in a domain.}
{\it Numer. Math.}
{\bf 128}
103--136
(2014).


\bibitem{NT20}
{Ngo, H.-L.} and {Taguchi, D.}
{Semi-implicit Euler--Maruyama approximation for non-colliding particle systems.}
{\it Ann. Appl. Probab.}
{\bf 30}(2)
673--705
(2020).

\bibitem{ReYo99}
{Revuz, D.} and {Yor, M.}
{\it Continuous Martingales and {B}rownian Motion, Third Edition.}
Springer
(1999).

\bibitem{RoSh}
{Rogers, L.~C.~G.} and {Shi, Z.}
{Interacting Brownian particles and the Wigner law.}
{\it Probab. Theory Related Fields.}
{\bf 95}(4)
555--570
(1993).

\bibitem{Ro98}
{R{\"o}sler, M.}
{Generalized Hermite polynomials and the heat equation for Dunkl operators.}
{\it Commun. Math. Phys}
{\bf 192}
519--541
(1998).


\bibitem{RoVo98}
{R{\"o}sler, M.} and {Voit, M.}
{Markov processes related with Dunkl operators.}
{\it Adv,  Appl. Math.}
{\bf 21}(4)
575--643
(1998).

\bibitem{Sa13}
{Sabanis, S.}
{A note on tamed Euler approximations.}
{\it Electron. Commun. Probab.}
{\bf 18}(47)
1--10
(2013).

\bibitem{Sa16}
{Sabanis, S.}
{Euler approximations with varying coefficients: the case of superlinearly growing diffusion coefficients.}
{\it Ann. Appl. Probab.}
{\bf 25}(4)
2083--2105
(2016).

\bibitem{Sc07}
{Schapira, B.}
{The Heckman--Opdam Markov processes.}
{\it Probab. Theory Related Fields}
{\bf 138}(3--4)
495--519
(2007).

\bibitem{Sk65}
{Skorokhod, A.~V.}
{\it Studies in the Theory of Random Processes.}
Addison-Wesley
(1965).


\bibitem{Ve81}
{Veretennikov, A.~Yu.}
{On strong solution and explicit formulas for solutions of stochastic integral equations.}
{\it Math. USSR Sb.}
{\bf 39}
387--403
(1981).


\bibitem{Yor80}
{Yor, M.}
{Loi de l'indice du lacet brownien, et distribution de {H}artman-{W}atson.}
{\it Z. Wahrsch. Verw. Gebiete}
{\bf 53}(1)
71--95
(1980).

\bibitem{Zv}
{Zvonkin, A.~K.}
{A transformation of the phase space of a diffusion process that removes the drift.}
{\it Math. USSR Sbornik}
{\bf 22}
129--148
(1974).

\end{thebibliography}
\end{document}